\documentclass{amsart}
\usepackage{tikz,tikz-cd}
\usepackage{url}
\usepackage{graphicx}
\usepackage{graphicx}
\usepackage{color}
\usepackage{times}
\usepackage{cite}
\usepackage{enumerate,latexsym}
\usepackage{latexsym}
\usepackage{amsmath,amssymb}
\usepackage{graphicx}
\usepackage{amsthm}
\usepackage{verbatim}
\newtheorem{thm}{Theorem}[section]
\newtheorem*{theorem*}{Theorem}
\newtheorem*{acknowledgement*}{Acknowledgement}
\newtheorem{cor}[thm]{Corollary}

\newtheorem{lem}[thm]{Lemma}
\newtheorem{prop}[thm]{Proposition}
\theoremstyle{definition}
\newtheorem{defn}[thm]{Definition}
\theoremstyle{remark}
\newtheorem{rem}[thm]{Remark}

\numberwithin{equation}{section}

\newcommand{\abs}[1]{\left\vert#1\right\vert}
\newcommand{\set}[1]{\left\{#1\right\}}
\newcommand{\Real}{\mathbb R}

\newcommand{\func}[1]{\ensuremath{\mathop{\mathrm{#1}}} }

\newcommand{\dist}[0]{\mathrm{dist}}

\newcommand{\spt}[0]{\func{spt}}

\newcommand{\sing}[0]{\func{sing}}
\newcommand{\reg}[0]{\func{reg}}
\newcommand{\xX}[0]{\mathbf{x}}

\newcommand{\yY}[0]{\mathbf{y}}

\newcommand{\nN}[0]{\mathbf{n}}

\newcommand{\IV}{\mathbf{IV}}
\newcommand{\IM}{\mathcal{IM}}

\newcommand{\M}{\mathcal{M}}
\newcommand{\OO}{\mathbf{0}}
\newcommand{\sstar}{\star\star}
\newcommand{\rstar}[1]{(\star_{#1})}
\newcommand{\rsstar}[1]{(\sstar_{#1})}

\title{Topology of Closed Hypersurfaces of Small Entropy }
\author{Jacob Bernstein}
\address{Department of Mathematics, Johns Hopkins University, 3400 N. Charles Street, Baltimore, MD 21218}
\email{bernstein@math.jhu.edu}
\author{Lu Wang}
\address{Department of Mathematics, University of Wisconsin, 480 Lincoln Drive, Madison, WI 53706}
\email{luwang@math.wisc.edu}
\thanks{The first author was partially supported by the NSF Grant DMS-1307953. The second author was partially supported by the Chapman Fellowship of Imperial College London and NSF Grant DMS-1406240}

\begin{document}
\begin{abstract}
We use a weak mean curvature flow together with a surgery procedure to show that all closed hypersurfaces in $\Real^4$ with entropy less than or equal to that of $\mathbb{S}^2\times \Real$, the round cylinder in $\Real^4$, are diffeomorphic to $\mathbb{S}^3$.
\end{abstract}

\maketitle

\section{Introduction} \label{Intro}
If $\Sigma$ is a {hypersurface}, that is, a smooth properly embedded codimension-one submanifold of $\Real^{n+1}$, then the 
\emph{Gaussian surface area} of $\Sigma$ is
\begin{equation}
F[\Sigma]=\int_{\Sigma}\Phi\, d\mathcal{H}^{n}=(4\pi)^{-\frac{n}{2}}\int_{\Sigma}e^{-\frac{|\xX|^2}{4}} d\mathcal{H}^n,
\end{equation}
where $\mathcal{H}^n$ is $n$-dimensional Hausdorff measure. Following Colding-Minicozzi \cite{CMGenMCF}, define the \emph{entropy} of $\Sigma$ to be
\begin{equation*}
\lambda[\Sigma]=\sup_{(\yY,\rho)\in\Real^{n+1}\times\Real^+}F[\rho\Sigma+\yY].
\end{equation*}
That is, the entropy of $\Sigma$ is the supremum of the Gaussian surface area over all translations and dilations of $\Sigma$. Observe that the entropy of a hyperplane is one. In \cite{BernsteinWang}, we show that, for $2\leq n\leq 6$, the entropy of a closed (i.e. compact and without boundary) hypersurface in $\Real^{n+1}$ is uniquely (modulo translations and dilations) minimized by $\mathbb{S}^n$, the unit sphere centered at the origin. This verifies a conjecture of Colding-Ilmanen-Minicozzi-White \cite[Conjecture 0.9]{CIMW} (cf. \cite{KetoverZhou}). We further show, in \cite[Corollary 1.3]{BernsteinWang2}, that surfaces in $\Real^3$ of small entropy are topologically rigid. That is, if $\Sigma$ is a closed surface in $\Real^3$ and $\lambda[\Sigma]\leq\lambda[\mathbb{S}^1\times\Real]$, then $\Sigma$ is diffeomorphic to $\mathbb{S}^2$. 
 
In this article, we use a weak mean curvature flow (see \cite{ES1,ES2, ES3, ES4} and \cite{CGG}) to obtain  new topological rigidity for closed hypersurfaces in $\Real^4$ of small entropy. This generalizes a result of Colding-Ilmanen-Minicozzi-White \cite{CIMW} for closed self-shrinkers to arbitrary closed hypersurfaces and contrasts with the methods of both \cite{CIMW} and \cite[Corollary 1.3]{BernsteinWang2}, which both use only the classical mean curvature flow.  
\begin{thm}\label{MainTopThm}
If $\Sigma\subset\Real^{4}$ is a closed hypersurface with $\lambda[\Sigma]\leq\lambda[\mathbb{S}^2\times\Real]$, then $\Sigma$ is diffeomorphic to $\mathbb{S}^3$.
\end{thm}

One of the key ingredients in the proof of Theorem \ref{MainTopThm} is a refinement of \cite[Theorem 0.1]{BernsteinWang2} about the topology of asymptotically conical self-shrinkers of small entropy. Recall, a hypersurface $\Sigma$ is said to be \emph{asymptotically conical} if it is smoothly asymptotic to a regular cone; i.e., $\lim_{\rho\to 0} \rho\Sigma= \mathcal{C} (\Sigma)$ in $C^{\infty}_{loc}(\Real^{n+1}\setminus\set{\OO})$ for $\mathcal{C} (\Sigma)$ a regular cone. A \emph{self-shrinker}, $\Sigma$, is a hypersurface that satisfies
\begin{equation}\label{SSEqn}
\mathbf{H}_\Sigma+\frac{\xX^\perp}{2}=\mathbf{0},
\end{equation}
where $\mathbf{H}_\Sigma=-H_{\Sigma} \nN_\Sigma=\Delta_\Sigma \xX$ is the mean curvature vector of $\Sigma$ and $\xX^\perp$ is the normal component of the position vector.   Let us denote the set of self-shrinkers in $\Real^{n+1}$ by $\mathcal{S}_n$ and the set of asymptotically conical self-shrinkers by $\mathcal{ACS}_n$. Self-shrinkers generate solutions to the mean curvature flow that move self-similarly by scaling. That is, if $\Sigma\in\mathcal{S}_n$, then
\begin{equation*}
\set{\Sigma_t}_{t\in(-\infty,0)}=\set{\sqrt{-t}\, \Sigma}_{t\in(-\infty,0)}
\end{equation*}
moves by mean curvature. Important examples are the maximally symmetric self-shrinking cylinders with $k$-dimensional spine,
\begin{equation*}
\mathbb{S}^{n-k}_{*}\times\Real^k=\set{(\xX,\yY)\in\Real^{n-k+1}\times\Real^k=\Real^{n+1}: |\xX|^2=2(n-k)},
\end{equation*}
where $0\leq k\leq n$. As $\mathbb{S}^{n-k}_{*}\times\Real^k$ are self-shrinkers, their Gaussian surface area and entropy agree (cf. \cite[Lemma 7.20]{CMGenMCF}).  That is,
\begin{equation*}
\lambda_n=\lambda [\mathbb{S}^n]=F[\mathbb{S}_*^n]=F[\mathbb{S}_*^n\times\Real^l]=\lambda[\mathbb{S}^n\times\Real^l].
\end{equation*}
Hence, a computation of Stone \cite{Stone}, gives that
\begin{equation*}
2>\lambda_1>\frac{3}{2}>\lambda_2>\ldots>\lambda_n>\ldots\to\sqrt{2}.
\end{equation*}

\begin{thm}
\label{MainACSThm}
Let $\Sigma\in \mathcal{ACS}_n$ for $n\geq 2$. If $\lambda[\Sigma]\leq\lambda_{n-1}$, then $\Sigma$ is contractible and $\mathcal{L} (\Sigma)$, the link of the asymptotic cone $\mathcal{C} (\Sigma)$, is a homology $(n-1)$-sphere. 
\end{thm}
\begin{rem}
	We always consider homology with integer coefficients.
\end{rem}
For $n=3$, the classification of surfaces and Alexander's theorem \cite{Alexander} gives
\begin{cor}
\label{ACSDim3Cor}
Let $\Sigma\in \mathcal{ACS}_3$. If $\lambda[\Sigma]\leq\lambda_{2}$, then $\Sigma$ is diffeomorphic to $\Real^3$.
\end{cor}
To prove Theorem \ref{MainTopThm} we first establish a topological decomposition, i.e., Theorem \ref{CondSurgThm}, constructed from the weak mean curvature flow associated to $\Sigma$. Together with Corollary \ref{ACSDim3Cor} this allows one to perform a surgery procedure which immediately gives the result.  Both these steps require $n=3$.  For $n\geq 4$, one can use Theorem \ref{MainACSThm} and this surgery procedure to show a (strictly weaker) extension of Theorem \ref{MainTopThm} valid in any dimension where the two hypotheses below are satisfied. These hypotheses ensure the existence of  topological decomposition.  Specifically, they ensure that if the entropy of an initial hypersurface is small enough, then tangent flows at all singularities are modeled by self-shrinkers that are either closed or asymptotically conical. 

In order to state these hypotheses, first let $\mathcal{S}_n^*$ denote the set of non-flat elements of $\mathcal{S}_n$ and, for any $\Lambda>0$, let 
$$
\mathcal{S}_n(\Lambda)=\set{\Sigma\in \mathcal{S}_n: \lambda[\Sigma]<\Lambda} \mbox{ and } \mathcal{S}_n^*(\Lambda)=\mathcal{S}^*_n \cap \mathcal{S}_n(\Lambda).
$$
Next, let $\mathcal{RMC}_n$ denote the space of \emph{regular minimal cones} in $\Real^{n+1}$, that is $\mathcal{C}\in \mathcal{RMC}_n$ if and only if it is a proper subset of $\Real^{n+1}$ and $\mathcal{C}\backslash\set{\OO}$ is a hypersurface in $\Real^{n+1}\backslash\set{\OO}$ that is invariant under dilation about $\OO$ and with vanishing mean curvature. Let $\mathcal{RMC}_n^*$ denote the set of non-flat elements of $\mathcal{RMC}_n$ -- i.e., cones whose second fundamental forms do not identically vanish. For any $\Lambda>0$, let 
$$
\mathcal{RMC}_n(\Lambda)=\set{\mathcal{C}\in \mathcal{RMC}_n: \lambda[\mathcal{C}]< \Lambda} \mbox{ and } \mathcal{RMC}_n^*(\Lambda)=\mathcal{RMC}^*_n \cap \mathcal{RMC}_n(\Lambda).
$$

Let us now fix a dimension $n\geq 3$ and a value $\Lambda>1$.  
The first hypothesis is 
\begin{equation} \label{Assump1}
\mbox{For all $3\leq k\leq n$, }\mathcal{RMC}_{k}^*(\Lambda)=\emptyset \tag{$\star_{n,\Lambda}$}.
\end{equation}
Observe that all regular minimal cones in $\mathbb{R}^2$ consist of unions of rays and so $\mathcal{RMC}^*_1=\emptyset$. Likewise, as great circles are the only geodesics in $\mathbb{S}^2$, $\mathcal{RMC}_2^*=\emptyset$.  The second hypothesis is
\begin{equation} \label{Assump2}
\mathcal{S}_{n-1}^*(\Lambda) =\emptyset. \tag{$\sstar_{n,\Lambda}$}
\end{equation}
Obviously this holds only if $\Lambda\leq\lambda_{n-1}$.
We then show the following conditional result:
\begin{thm}\label{MainCondThm}
Fix $n\geq 4$ and $\Lambda\in (\lambda_{n}, \lambda_{n-1}]$. If \eqref{Assump1} and \eqref{Assump2} both hold and $\Sigma$ is a closed hypersurface in $\Real^{n+1}$ with $\lambda[\Sigma]\leq\Lambda$, then $\Sigma$ is a homology $n$-sphere.
\end{thm}
\begin{rem} \label{LowBndLam}
If \eqref{Assump1} and \eqref{Assump2} hold for $\Lambda\leq\lambda_{n}$, then it follows from Huisken's monotonicity formula and the results of \cite{BernsteinWang} and \cite{CIMW} that there does \emph{not} exist a closed hypersurface $\Sigma$ so that $\lambda[\Sigma]\leq\Lambda$ unless $\Lambda=\lambda_n$ and $\Sigma$ is a round sphere. Thus, we require $\Lambda>\lambda_n$ in order to make Theorem \ref{MainCondThm} non-trivial. 
\end{rem}

For general $n$ and $\Lambda\in (\lambda_n, \lambda_{n-1}]$, neither the validity of \eqref{Assump1} nor that of \eqref{Assump2} is known. However, both can be established for $n=3$ and $\Lambda=\lambda_2$. First, as part of their proof of the Willmore conjecture, Marques-Neves gave a lower bound on the density of non-trivial regular minimal cones in $\Real^4$. In particular, it follows from \cite[Theorem B]{MarquesNeves} that if $\mathcal{C}\in\mathcal{RMC}_3^*$, then $\lambda[\mathcal{C}] >\lambda_2$ and so $\rstar{3,\lambda_2}$ holds. Furthermore, it follows from \cite[Corollary 1.2]{BernsteinWang2} that $\mathcal{S}_2^*(\lambda_2)=\emptyset$ and so $\rsstar{3,\lambda_2}$ holds.

For $n\geq 4$, some partial results suggest that \eqref{Assump1}  and \eqref{Assump2} hold for $\Lambda=\lambda_{n-1}$. For instance, Ilmanen-White \cite[Theorem 1*]{Ilmanen-White}, have shown that if $\mathcal{C}\in \mathcal{RMC}_n^*$ and is area-minimizing and topologically non-trivial, then $\lambda[\mathcal{C}]\geq \lambda_{n-1}$. Additionally, \cite[Theorem 0.1]{CIMW} says that the self-shrinking sphere has the lowest entropy among all compact self-shrinkers and \cite[Conjecture 0.10]{CIMW} posits that $\rsstar{n,\lambda_{n-1}}$ holds for $n\leq 7$.
It is important to note that there exist many topologically trivial elements of $\mathcal{RMC}_n^*$. Indeed, the work of Hsiang \cite{Hsiang1, Hsiang2} and Hsiang-Sterling \cite{HsiangSterling}, shows that there exist topologically trivial elements of $\mathcal{RMC}_n^*$ for $n=5,7$ and for all even $n\geq 4$.   

The paper is organized as follows. In Section \ref{NotationSec}, we introduce notation and recall some basic facts about the mean curvature flow. In Section \ref{RegularitySec}, we show regularity of self-shrinking measures of low entropy. In Section \ref{SingularitySec}, we study the structure of the singular set for weak mean curvature flows of small entropy. Importantly, we give a topological decomposition, Theorem \ref{CondSurgThm}, of the regular part of the flow which is the basis of the surgery procedure. In Section \ref{SharpeningSec}, we prove Theorem \ref{MainACSThm} and Corollary \ref{ACSDim3Cor}.  Finally, in Section \ref{SurgerySec}, we carry out the surgery procedure and prove Theorems \ref{MainTopThm} and \ref{MainCondThm}.

\section{Notation and Background} \label{NotationSec}
In this section, we fix notation for the rest of the paper and recall some background on mean curvature flow.  Experts should feel free to consult this section only as needed.

\subsection{Singular hypersurfaces}
We will use results from \cite{Ilmanen1} on weak mean curvature flows. For this reason, we follow the notation of \cite{Ilmanen1} as closely as possible. 

Denote by
\begin{itemize}
 \item $\M(\Real^{n+1})=\set{\mu: \mu\mbox{ is a Radon measure on $\Real^{n+1}$}}$ (see \cite[Section 4]{Simon});
 \item $\IM_k(\Real^{n+1})=\set{\mu: \mu\mbox{ is an integer $k$-rectifiable Radon measure on $\Real^{n+1}$}}$ (see \cite[Section 1]{Ilmanen1});
 \item $\IV_k(\Real^{n+1})=\set{V: V\mbox{ is an integer rectifiable $k$-varifold on $\Real^{n+1}$}}$ (see \cite[Section 1]{Ilmanen1} or \cite[Chapter 8]{Simon}).
\end{itemize}
The space $\M(\Real^{n+1})$ is given the weak* topology. That is,
$$
 \mu_i\to\mu\iff\int f\, d\mu_i\to\int f \, d\mu\mbox{  for all $f\in C^0_c(\Real^{n+1})$}.
$$
And the topology on $\IM_k(\Real^{n+1})$ is the subspace topology induced by the natural inclusion into $\M(\Real^{n+1})$. For the details of the topologies considered on $ \IV_k(\Real^{n+1})$, we refer to \cite[Section 1]{Ilmanen1} or \cite[Chapter 8]{Simon}. There are natural bijective maps
$$
V: \IM_k(\Real^{n+1})\to \IV_k(\Real^{n+1}) \mbox{ and } \mu:\IV_k(\Real^{n+1})\to\IM_k(\Real^{n+1}).
$$
The second map is continuous, but the first is not. Henceforth, write $V(\mu)=V_\mu$ and $\mu(V)=\mu_V$. 

If $\Sigma\subset\Real^{n+1}$ is a $k$-dimensional smooth properly embedded submanifold, we denote by $\mu_\Sigma=\mathcal{H}^k\lfloor\Sigma\in\IM_k(\Real^{n+1})$. Given $(\yY,\rho)\in\Real^{n+1}\times\Real^+$ and $\mu\in\IM_k(\Real^{n+1})$, we define the rescaled measure $\mu^{\yY,\rho}\in\IM_k(\Real^{n+1})$ by
$$
 \mu^{\yY,\rho}(\Omega)=\rho^{k}\mu\left(\rho^{-1}\Omega+\yY\right).
$$
This is defined so that if $\Sigma$ is a $k$-dimensional smooth properly embedded submanifold, then 
$$
\mu^{\yY,\rho}_\Sigma=\mu_{\rho (\Sigma-\yY)}. 
$$
One of the defining properties of $\mu\in \IM_k(\Real^{n+1})$ is that for $\mu$-a.e. $\xX\in\Real^{n+1}$, there is an integer 
$\theta_\mu(\xX)$ so that
$$
 \lim_{\rho\to \infty}\mu^{\xX,\rho}=\theta_\mu(\xX)\mu_{P},
$$
where $P$ is a $k$-dimensional plane through the origin. When such $P$ exists, we denote it by $T_{\xX} \mu$ the \emph{approximate tangent plane at $\xX$}. The value $\theta_\mu(\xX)$ is the \emph{multiplicity of $\mu$ at $\xX$} and by definition, $\theta_\mu(\xX)\in\mathbb{N}$ for $\mu$-a.e. $\xX$. Notice that if $\mu=\mu_{\Sigma}$, then $T_{\xX}\mu=T_{\xX}\Sigma$ and $\theta_\mu(\xX)=1$. Given a $\mu\in \IM_n(\Real^{n+1})$, set
$$
\reg(\spt(\mu))=\set{\xX\in\spt(\mu): \exists\rho>0 \mbox{ s.t. $B_\rho(\xX)\cap\spt(\mu)$ is a hypersurface}},
$$
and $\sing(\spt(\mu))=\spt(\mu)\setminus\reg(\spt(\mu))$. Here $B_\rho(\xX)$ is the open ball in $\Real^{n+1}$ centered at $\xX$ with radius $\rho$. Likewise,
$$
\reg(\mu)=\set{\xX\in \reg(\spt(\mu)): \theta_\mu(\xX)=1}  \mbox{ and }  \sing(\mu)=\spt(\mu)\setminus\reg(\mu).
$$

For $\mu\in\IM_n(\Real^{n+1})$, we extend the definitions of $F$ and $\lambda$ in the obvious manner, namely,
$$
F[\mu]=F[V_\mu]=\int \Phi\, d\mu  \mbox{ and }  \lambda[\mu]=\lambda[V_\mu]=\sup_{(\yY,\rho)\in\Real^{n+1}\times\Real^+} F[\mu^{\yY,\rho}].
$$

\subsection{Gaussian densities and tangent flows}\label{Brakkef}
Historically, the first weak mean curvature flow was the measure-theoretic flow introduced by Brakke \cite{B}. This flow is called a \emph{Brakke flow}. Brakke's original definition considered the flow of varifolds. We use the (slightly stronger) notion introduced by Ilmanen \cite[Definition 6.3]{Ilmanen1}. For our purposes, the Brakke flow has two important roles. The first is the fact that Huisken's monotonicity formula \cite{Huisken} holds also for Brakke flows; see \cite[Lemma 7]{Ilmanen2}. The second is the powerful regularity theory of Brakke \cite{B} for such flows. In particular, we will often refer to White's version of Brakke's local regularity theorem  \cite{WhiteReg}.  We emphasize that White's argument is valid only for a special class of Brakke flows, but that all Brakke flows considered in this paper are within this class.

A consequence of Huisken's monotonicity formula is that if a Brakke flow $\mathcal{K}=\set{\mu_t}_{t\geq t_0}$ has bounded area ratios, then $\mathcal{K}$ has a well-defined \emph{Gaussian density} at every point $(\yY,s)\in\Real^{n+1}\times (t_0,\infty)$ given by
$$
\Theta_{(\yY,s)}(\mathcal{K})=\lim_{t\to s^-}\int\Phi_{(\yY,s)}(\xX,t)\, d\mu_t(\xX),
$$
where 
$$
\Phi_{(\yY,s)} (\xX,t)=(4\pi)^{-\frac{n}{2}} e^{\frac{|\xX-\yY|^2}{4(t-s)}}.
$$
Furthermore, the Gaussian density is upper semi-continuous.

Combining the compactness of Brakke flows (cf.  \cite[7.1]{Ilmanen1}) with the monotonicity formula, one establishes the existence of tangent flows. For a Brakke flow $\mathcal{K}=\set{\mu_t}_{t\geq t_0}$ and a point $(\yY,s)\in\Real^{n+1}\times(t_0,\infty)$, define a new Brakke flow 
$$
\mathcal{K}^{(\yY,s),\rho}=\set{\mu_t^{(\yY,s),\rho}}_{t\geq\rho^2(t_0-s)},
$$
where
$$
\mu_t^{(\yY,s),\rho}=\mu_{s+\rho^{-2}t}^{\yY,\rho}.
$$
\begin{defn}
Let $\mathcal{K}=\set{\mu_t}_{t\geq t_0}$ be an integral Brakke flow with bounded area ratios. A non-trivial Brakke flow $\mathcal{T}=\set{\nu_t}_{t\in\Real}$ is a \emph{tangent flow} to $\mathcal{K}$ at $(\yY,s)\in\Real^{n+1}\times(t_0,\infty)$, if there is a sequence $\rho_i\to\infty$ so that $\mathcal{K}^{(\yY,s),\rho_i}\to\mathcal{T}$. Denote by $\mathrm{Tan}_{(\yY,s)}\mathcal{K}$ the set of tangent flows to $\mathcal{K}$ at $(\yY,s)$.
\end{defn}
The monotonicity formula implies that any tangent flow is backwardly self-similar.
\begin{thm}[{\cite[Lemma 8]{Ilmanen2}}]\label{BlowupsThm}
Given an integral Brakke flow $\mathcal{K}=\set{\mu_t}_{t\geq t_0}$ with bounded area ratios, a point $(\yY,s)\in\Real^{n+1}\times({t_0},\infty)$ with $\Theta_{(\yY,s)}(\mathcal{K})\geq 1$, and a sequence $\rho_i\to\infty$, there exists a subsequence $\rho_{i_j}$ and a $\mathcal{T}\in\mathrm{Tan}_{(\yY,s)}\mathcal{K}$ so that $\mathcal{K}^{(\yY,s),\rho_{i_j}}\to\mathcal{T}$.

Furthermore, $\mathcal{T}=\set{\nu_t}_{t\in\Real}$ is backwardly self-similar with respect to parabolic rescaling about $(\OO,0)$. That is, for all $t<0$ and $\rho>0$, 
$$
\nu_t=\nu_t^{(\OO,0),\rho}. 
$$
Moreover, $V_{\nu_{-1}}$ is a stationary point of the $F$ functional and 
$$
\Theta_{(\yY,s)}(\mathcal{K})=F[\nu_{-1}].
$$
\end{thm}

\subsection{Level set flows and boundary motions}
We will also need a set-theoretic weak mean curvature flow called the level-set flow. This flow was first studied in the context of numerical analysis by Osher-Sethian \cite{OS}. The mathematical theory was developed by Evans-Spruck \cite{ES1,ES2,ES3,ES4} and Chen-Giga-Goto \cite{CGG}. For our purposes, it has the important advantages of being uniquely defined and satisfying a maximum principle. 

A technical feature of the level-set flow is that the level sets ${L}(\Gamma_0)=\set{\Gamma_t}_{t\geq 0}$ may develop non-empty interiors for positive times. This phenomena is called fattening and is unavoidable for certain initial sets $\Gamma_0$ and is closely related to non-uniqueness phenomena of weak solutions of the flow.  We say $L(\Gamma_0)$ is \emph{non-fattening}, if each $\Gamma_t$ has no interior.  
It is relatively straightforward to see that the non-fattening condition is generic; see for instance \cite[Theorem 11.3]{Ilmanen1}.

In \cite{Ilmanen1}, Ilmanen synthesized both notions of weak flow. In particular, he showed that for a large class of initial sets, there is a canonical way to associate a Brakke flow to the level-set flow, and observed that this allows, among other things, for the application of Brakke's partial regularity theorem. For our purposes, it is important that the Brakke flow constructed does not vanish gratuitously. A similar synthesis may be found in \cite{ES4}.  The result we need is the following:
\begin{thm}[{\cite[Theorem 11.4]{Ilmanen1}}] \label{UnitdensityThm}
If $\Sigma_0$ is a closed hypersurface in $\Real^{n+1}$ and the level-set flow ${L}(\Sigma_0)$ is non-fattening, then there is a set $E\subset \Real^{n+1}\times \Real$ and a Brakke flow $\mathcal{K}=\set{\mu_t}_{t\geq 0}$ so that: 
\begin{enumerate}
 \item $E=\set{(\xX,t): u(\xX,t)>0}$, where $u$ solves the level set flow equation with initial data $u_0$ that satisfies $E_0=\set{\xX: u_0(\xX)>0}$ and $\partial E_0=\set{\xX: u_0(\xX)=0}=\Sigma_0$;
 \item each $E_t=\set{\xX: (\xX,t)\in E}$ is of finite perimeter and $\mu_t=\mathcal{H}^n\lfloor\partial^\ast E_t$, where $\partial^* E_t$ is the reduced boundary of $E_t$.
\end{enumerate}
\end{thm}

\section{Regularity of Self-Shrinking Measures of Small Entropy} \label{RegularitySec}
We establish some regularity properties of self-shrinking measures of small entropy when $n\geq 3$. We restrict to $n\geq 3$ in order to avoid certain technical complications coming from the fact that $\lambda_1>\frac{3}{2}$. 

\subsection{Self-shrinking measures}
We will need a singular analog of $\mathcal{S}_n$.  To that end, we define the set of self-shrinking measures on $\Real^{n+1}$ by
\begin{equation*}
\mathcal{SM}_n=\set{\mu\in\IM_n(\Real^{n+1}): V_\mu\mbox{ is stationary for the ${F}$ functional}, \spt(\mu)\neq\emptyset}.
\end{equation*}
Clearly, if $\Sigma\in \mathcal{S}_n$, then $\mu_\Sigma\in\mathcal{SM}_n$.  There are many examples of singular self-shrinkers. For instance,  any element of $\mathcal{C}\in\mathcal{RMC}_n$ satisfies $\mu_{\mathcal{C}}=\mathcal{H}^n\lfloor\mathcal{C}\in \mathcal{SM}_n$. 
For $\mu\in\mathcal{SM}_n$, we define the \emph{associated Brakke flow} $\mathcal{K}=\set{\mu_t}_{t\in\Real}$ by
\begin{equation*}
\mu_t= \left \{ \begin{array}{cc} 
                               0 & t\geq 0 \\
                               \mu^{\OO, \sqrt{-t}} & t<0.
                       \end{array} \right. 
\end{equation*}
One can verify that this is a Brakke flow.
Given $\Lambda>0$, set
\begin{equation*}
\mathcal{SM}_n(\Lambda)=\set{\mu\in\mathcal{SM}_n: \lambda[\mu]<\Lambda} \mbox{ and } \mathcal{SM}_n[\Lambda]=\set{\mu\in\mathcal{SM}_n: \lambda[\mu]\leq \Lambda}.
\end{equation*}

\subsection{Regularity and asymptotic properties of self-shrinking measures of small entropy}
A $\mu\in\IM_n(\Real^{n+1})$ is a \emph{cone}, if $\mu^{\OO,\rho}=\mu$. Likewise, $\mu\in\IM_n(\Real^{n+1})$ \emph{splits off a line}, if, up to an ambient rotation of $\Real^{n+1}$, $\mu=\hat{\mu}\times\mu_\Real$ for $\hat{\mu}\in\IM_{n-1}(\Real^{n})$. Observe that if $\mu\in\mathcal{SM}_n$ is a cone, then $V_\mu$ is stationary (for area). Similarly, if $\mu\in\mathcal{SM}_n$ splits off a line, then $\hat{\mu}\in\mathcal{SM}_{n-1}$ and $\lambda[\mu]=\lambda[\hat{\mu}]$.

Standard dimension reduction arguments give the following:
\begin{lem} \label{DimRedConeLem}
Fix $n\geq 3$ and  $\Lambda\leq 3/2$ and suppose that \eqref{Assump1} holds. If $\mu\in\mathcal{SM}_n(\Lambda)$ is a cone, then $\mu=\mu_P$ for some hyperplane $P$.
\end{lem}
\begin{proof}
 We will prove this by showing that if \eqref{Assump1} holds, then for all $3\leq m\leq n$, if $\mu\in\mathcal{SM}_m(\Lambda)$ is a cone, then $\mu=\mu_{P}$ for a hyperplane $P$ in $\mathbb{R}^{m+1}$.

We proceed by induction on $m$. When $m=3$, note that $\Lambda\leq\frac{3}{2}$ and so we have that $\mu=\mu_{\mathcal{C}}$ for some $\mathcal{C}\in \mathcal{RMC}_3$ by \cite[Proposition 4.2]{BernsteinWang}. Hence, by the assumption that $\mathcal{RMC}_3^*(\Lambda)=\emptyset$, we must have that $\mathcal{C}$ is a hyperplane through the origin.  
To complete the induction argument, we observe that it suffices to show that if $\mu\in \mathcal{SM}_m(\Lambda)$ is a cone, then $\mu=\mu_{\mathcal{C}}$ for some $\mathcal{C}\in \mathcal{RMC}_m(\Lambda)$.  Indeed, such a $\mathcal{C}$ must be a hyperplane because \eqref{Assump1} holds and so, by definition, $\mathcal{RMC}^*_m(\Lambda)=\emptyset$ for $3\leq m \leq n$.

To complete the proof, we argue by contradiction. Suppose that $\spt(\mu)$ is not a regular cone. Then there is a point $\yY\in\sing(\mu)\setminus\set{\OO}.$  As $V_\mu$ is stationary, and $\mu\in\IM_m$ with $\lambda[\mu]<\Lambda$, we may apply Allard's integral compactness theorem (see \cite[Theorem 42.7 and Remark 42.8]{Simon}) to conclude that there exists a sequence $\rho_i\to\infty$ so that $\mu^{\yY,\rho_i}\to\nu$ and $V_\nu$ is a stationary integral varifold. Moreover, it follows from the monotonicity formula \cite[Theorem 17.6]{Simon} that $\nu$ is a cone; see also \cite[Theorem 19.3]{Simon}.

As $\mu$ is a cone, $\nu$ splits off a line. That is, $\nu=\hat{\nu}\times\mu_\Real$, where $\hat{\nu}\in\mathcal{IM}_{m-1}$ and $V_{\hat{\nu}}$ is a stationary  cone and so $\hat{\nu}\in \mathcal{SM}_{m-1}$. Moreover, by the lower semi-continuity of entropy, 
$$
\lambda[\hat{\nu}]=\lambda[\hat{\nu}\times\mu_\Real]\leq\lambda[\mu]<\Lambda.
$$
Thus, it follows from the induction hypotheses that $\hat{\nu}=\mu_{\hat{P}}$, where $\hat{P}$ is a hyperplane in $\Real^m$ and so $V_\nu$ is a multiplicity-one hyperplane. Hence, by Allard's regularity theorem (see \cite[Theorem 24.2]{Simon}), $\yY\in\reg(\mu)$, giving a contradiction. Therefore, $\mu=\mu_{\mathcal{C}}$ for a $\mathcal{C}\in\mathcal{RMC}_m(\Lambda)$.
\end{proof}

As a consequence, we obtain regularity for elements of $\mathcal{SM}_n(\Lambda)$ under the hypothesis that \eqref{Assump1} holds.
\begin{prop} \label{CondRegProp}
Fix $n\geq 3$ and  $\Lambda\leq 3/2$ and suppose that \eqref{Assump1} holds. If $\mu\in \mathcal{SM}_n(\Lambda)$, then $\mu=\mu_\Sigma$ for some $\Sigma\in \mathcal{S}_n(\Lambda)$.
\end{prop}
\begin{proof}
Observe that for $\mu\in\mathcal{SM}_n(\Lambda)$, the mean curvature of $V_\mu$ is locally bounded by \eqref{SSEqn}. Following the same reasoning in the proof of Lemma \ref{DimRedConeLem}, given $\yY\in\sing (\mu)$, there exists a sequence $\rho_i\to\infty$ so that $\mu^{\yY,\rho_i}\to\nu$ and $V_\nu$ is a stationary cone and so $\nu\in \mathcal{SM}_{n}$. By the lower semi-continuity of entropy, $\lambda[\nu]\leq\lambda[\mu]<\Lambda$. Hence, together with Lemma \ref{DimRedConeLem}, it follows that $\sing(\mu)=\emptyset$. That is, $\spt(\mu)$ is a smooth submanifold of $\Real^{n+1}$ that, moreover, satisfies \eqref{SSEqn}.  Finally, the entropy bound on $\mu$ implies that $\mu(B_R)\leq C R^n$ for some $C>0$ and so, by \cite[Theorem 1.3]{CZProper}, $\spt(\mu)$ is proper.  That is, $\mu=\mu_\Sigma$ for some $\Sigma\in \mathcal{S}_n$.
\end{proof}

If, in addition, \eqref{Assump2} holds:
\begin{prop}\label{CondTotRegProp}
Fix $n\geq 3$ and $\Lambda\leq\Lambda_{n-1}$ and suppose that both \eqref{Assump1} and \eqref{Assump2} hold. If $\mu\in\mathcal{SM}_n(\Lambda)$, then $\mu=\mu_\Sigma$ for some $\Sigma\in\mathcal{S}_n(\Lambda)$, and either $\Sigma$ is diffeomorphic to $\mathbb{S}^n$ or $\Sigma\in\mathcal{ACS}_n$. 
\end{prop}
\begin{proof}
First observe that, by Proposition \ref{CondRegProp}, $\mu=\mu_\Sigma$ for some $\Sigma\in \mathcal{S}_n(\Lambda)$. If $\Sigma$ is closed, then it follows from \cite[Theorem 0.7]{CIMW} that $\Sigma$ is diffeomorphic to $\mathbb{S}^n$. On the other hand, if $\Sigma$ is not closed, then it is non-compact.

Let $\mathcal{K}=\set{\mu_t}_{t\in\Real}$ be the Brakke flow associated to $\mu$. Note that $\mu_t=\mu_{\sqrt{-t}\, \Sigma}$ for $t<0$. Let $\mathcal{X}=\set{\yY: \yY\neq\OO, \Theta_{(\yY,0)}(\mathcal{K})\geq 1}\subset \Real^{n+1}\setminus\set{\OO}$. As $\Sigma$ is non-compact, $\mathcal{X}$ is non-empty. Indeed, pick any sequence of points $\yY_i\in\Sigma$ with $|\yY_i|\to\infty$. The points $\hat{\yY}_i=|\yY_i|^{-1}\yY_i\in |\yY_i|^{-1}\Sigma$. Hence, $\Theta_{(\hat{\yY}_i, -|\yY_i|^{-2})}(\mathcal{K})\geq 1$. As the $\hat{\yY}_i$ are in a compact subset, up to passing to a subsequence and relabeling, $\hat{\yY}_i\to\hat{\yY}$, and so the upper semi-continuity of Gaussian density implies that $\Theta_{(\hat{\yY},0)}(\mathcal{K})\geq 1$.  

We next show that $\mathcal{X}$ is a regular cone. The fact that $\mathcal{X}$ is a cone readily follows from the fact that $\mathcal{K}$ is invariant under parabolic scalings. To see that $\sing(\mathcal{X})\subset \set{\OO}$, we note that, by \cite[Lemma 4.4]{BernsteinWang}, for any $\yY\in\mathcal{X}$ and $\mathcal{T}\in\mathrm{Tan}_{(\yY,0)}\mathcal{K}$, $\mathcal{T}=\set{\nu_t}_{t\in\Real}$ splits off a line. That is, up to an ambient rotation, $\nu_{t}=\hat{\nu}_t\times\mu_\Real$ with $\set{\hat{\nu}_t}_{t\in\Real}$ the Brakke flow associated to $\hat{\nu}_{-1}\in\mathcal{SM}_{n-1}(\Lambda)$.  Here we use the lower semi-continuity of entropy. Note that $\Lambda\leq\lambda_{n-1}<3/2$. Thus, by Proposition \ref{CondRegProp} and the hypothesis that $\rstar{n,\Lambda}$ holds, $\hat{\nu}_{-1}=\mu_\Gamma$ for $\Gamma\in\mathcal{S}_{n-1}(\Lambda)$. Hence, as we assume that \eqref{Assump2} holds, $\Gamma$ is a hyperplane through the origin. Therefore, it follows from Brakke's regularity theorem that, for $t<0$ close to $0$, $\spt(\mu_t)$ has uniformly bounded curvature near $\yY$ and so $\sqrt{-t}\, \Sigma\to\mathcal{X}$ in $C^\infty_{loc}\left(\Real^{n+1}\backslash \set{\OO}\right)$, concluding the proof.  
\end{proof}

As a consequence, we establish the following compactness theorem for asymptotically conical self-shrinkers of small entropy.
\begin{cor}\label{CpctnessACSCor}
Fix $n\geq 3$,  $\Lambda\leq\Lambda_{n-1}$, and $\epsilon_0>0$. If both \eqref{Assump1} and \eqref{Assump2} hold, then the set
$$
\mathcal{ACS}_n[\Lambda-\epsilon_0]=\set{\Sigma: \Sigma \in \mathcal{ACS}_n \mbox{ and } \lambda[\Sigma]\leq \Lambda-\epsilon_0}
$$ 
is compact in the $C^\infty_{loc}(\Real^{n+1})$ topology. 
\end{cor}
\begin{proof}
Consider a sequence $\Sigma_i\in \mathcal{ACS}_n[\Lambda-\epsilon_0]$ and let $\mu_i=\mu_{\Sigma_i}\in\mathcal{SM}_n[\Lambda-\epsilon_0]$. By the integral compactness theorem for $F$-stationary varifolds, up to passing to a subsequence, $\mu_i\to \mu$ in the sense of Radon measures. Moreover, by the lower semi-continuity of the entropy, $\mu\in \mathcal{SM}_n[\Lambda-\epsilon_0]$. Hence, by Proposition \ref{CondRegProp}, $\mu=\mu_\Sigma$ for $\Sigma\in \mathcal{S}_n[\Lambda-\epsilon_0]$ and so, by Allard's regularity theorem, $\Sigma_i\to \Sigma$ in $C^\infty_{loc}(\Real^{n+1})$. Finally, as each $\Sigma_i$ is non-compact and connected, so is $\Sigma$ and so, by Proposition \ref{CondTotRegProp}, $\Sigma\in \mathcal{ACS}_n[\Lambda-\epsilon_0]$, proving the claim.
\end{proof}

Recall that $\mathcal{C}(\Sigma)$ denotes the asymptotic cone of any $\Sigma\in \mathcal{ACS}_n$. Denote the link of the asymptotic cone by $\mathcal{L}(\Sigma)=\mathcal{C}(\Sigma)\cap \mathbb{S}^n$.
\begin{prop}\label{CpctnessLinksProp}
Fix $n\geq 3$,  $\Lambda \leq \lambda_{n-1}$, and $\epsilon_0>0$.  If both \eqref{Assump1} and \eqref{Assump2} hold, then the set
$$
\mathcal{L}_n[\Lambda-\epsilon_0]=\set{\mathcal{L}(\Sigma): \Sigma\in \mathcal{ACS}_n[\Lambda-\epsilon_0]}
$$
is compact in the $C^\infty(\mathbb{S}^n)$ topology.
\end{prop}
\begin{proof}
Consider a sequence $L_i\in \mathcal{L}_n[\Lambda-\epsilon_0]$ and let $\Sigma_i\in \mathcal{ACS}_n[\Lambda-\epsilon_0]$ be chosen so that $\mathcal{L}(\Sigma_i)=L_i$ (observe that the $\Sigma_i$ are uniquely determined by \cite[Theorem 1.3]{WaRigid}).  By Corollary \ref{CpctnessACSCor}, up to passing to a subsequence, $\Sigma_i\to \Sigma\in \mathcal{ACS}_n[\Lambda-\epsilon_0]$.  We claim that $L_i\to L=\mathcal{L}(\Sigma)$ in $C^\infty(\mathbb{S}^n)$.

To see this, let $\mu_i=\mu_{\Sigma_i}$ and $\mu=\mu_\Sigma$ be the corresponding elements of $\mathcal{SM}_n[\Lambda-\epsilon_0]$ and let $\mathcal{K}_i$ and $\mathcal{K}$ be the associated Brakke flows.  Clearly, $\mu_i\to \mu$ in the sense of measures.  Hence, by construction, the $\mathcal{K}_i$ converge in the sense of Brakke flows to ${\mathcal{K}}$.  Since
$$\mathcal{C}(\Sigma)=\set{\xX\in \Real^{n+1}: \Theta_{(\xX,0)}(\mathcal{K}) \geq 1}$$
and likewise for $\mathcal{C}(\Sigma_i)$, we have by Brakke's regularity theorem that $\mathcal{C}(\Sigma_i)\to \mathcal{C}(\Sigma)$ in $C^\infty_{loc}(\Real^{n+1}\backslash \set{\OO})$, that is $\mathcal{L}(\Sigma_i)\to \mathcal{L}(\Sigma)$ in $C^\infty(\mathbb{S}^n)$ as claimed.
\end{proof}

Let $B_R$ denote the open ball in $\mathbb{R}^{n+1}$ centered at the origin with radius $R$. Combining Corollary \ref{CpctnessACSCor} and Proposition \ref{CpctnessLinksProp} gives that
\begin{cor}\label{GraphCor}
Fix $n\geq 3$,  $\Lambda \leq \lambda_{n-1}$, and $\epsilon_0>0$. Suppose that \eqref{Assump1} and \eqref{Assump2} hold. There is an $R_0=R_0(n, \Lambda, \epsilon_0)$ and $C_0=C_0(n, \Lambda, \epsilon_0)$ so that if $\Sigma \in \mathcal{ACS}_n[\Lambda-\epsilon_0]$, then
\begin{enumerate}
\item $\Sigma\setminus\bar{B}_{R_0}$ is given by the normal graph of a smooth function $u$ over $\mathcal{C}(\Sigma)\setminus\Omega$, where $\Omega$ is a compact set, satisfying that for $p\in\mathcal{C}(\Sigma)\setminus\Omega$,
$$
\abs{\xX(p)} \abs{u(p)}+ \abs{\xX(p)}^2 \abs{\nabla_{\mathcal{C}(\Sigma)} u(p)}+\abs{\xX(p)}^3 \abs{\nabla_{\mathcal{C}(\Sigma)}^2 u(p)}\leq C_0;
$$
\item given $\delta>0$, there is a $\kappa\in (0,1)$ and $\mathcal{R}>1$ depending only on $n,\Lambda, \epsilon_0$ and $\delta$ so that if $p\in\Sigma\setminus B_{\mathcal{R}}$ and $r=\kappa |\xX (p)|$, then $\Sigma\cap B_r(p)$ can be written as a connected graph of a function $v$ over a subset of $T_p\Sigma$ with $|Dv|\leq\delta$. 
\end{enumerate} 
As such, for any $R\geq R_0$, $\Sigma\backslash B_R$ is diffeomorphic to $\mathcal{L}(\Sigma)\times [0, \infty)$.
\end{cor}
\begin{proof}
For any sequence $\Sigma_i\in\mathcal{ACS}_n[\Lambda-\epsilon_0]$, by Corollary \ref{CpctnessACSCor} and Proposition \ref{CpctnessLinksProp}, up to passing to a subsequence, $\Sigma_i\to\Sigma$ in $C^\infty_{loc} (\Real^{n+1})$ for some $\Sigma\in\mathcal{ACS}_n[\Lambda-\epsilon_0]$, and $\mathcal{L}(\Sigma_i)\to\mathcal{L}(\Sigma)$ in $C^\infty (\mathbb{S}^n)$. Let $\mathcal{K}_i$ and $\mathcal{K}$ be the associated Brakke flows to $\Sigma_i$ and $\Sigma$, respectively. As $\Sigma\in\mathcal{ACS}_n$, $\mathcal{K}\lfloor (B_2\setminus \bar{B}_1)\times [-1,0]$ is a smooth mean curvature flow. Furthermore, since $\mathcal{K}_i\to\mathcal{K}$, it follows from Brakke's local regularity theorem that $\Sigma_i$ have uniform curvature decay, more precisely, there exist $R, C>0$ so that for all $i$ and $p\in\Sigma_i\setminus B_R$, 
$$
\sum_{k=0}^2 \abs{\xX(p)}^{k+1}\abs{\nabla^k_{\Sigma_i} A_{\Sigma_i}(p)}\leq C,
$$
where $A_{\Sigma_i}$ is the second fundamental form of $\Sigma_i$. As the $\mathcal{C}(\Sigma_i)\to\mathcal{C}(\Sigma)$, by \cite[Lemma 2.2]{WaRigid} and \cite[Proposition 4.2]{BernsteinWang2}, there exist $R^\prime, C^\prime>0$ so that Items (1) and (2) in the statement hold for all $\Sigma_i$. This establishes the corollary by the arbitrariness of the $\Sigma_i$.
\end{proof}
Finally, we need the fact that closed self-shrinkers of small entropy have an upper bound on their extrinsic diameter.
\begin{prop}\label{DiamBoundProp}
Fix $n\geq 3$, $\Lambda\leq\lambda_{n-1}$, and $\epsilon_0>0$. Suppose that both \eqref{Assump1} and \eqref{Assump2} hold. Then there is a $R_D=R_D(n,\Lambda, \epsilon_0)$ so that if $\Sigma\in \mathcal{S}_n[\Lambda-\epsilon_0]$ is closed, then $\Sigma\subset \bar{B}_{R_D}$. 
\end{prop}
\begin{proof}
 We argue by contradiction. If this was not true, then there would be a sequence of $\Sigma_i\in  \mathcal{S}_n[\Lambda-\epsilon_0]$ with the property that there are points $p_i\in \Sigma_i$ with $|p_i|\to \infty$. In particular, for each $R>\sqrt{2n}$, there is an $i_0=i_0(R)$ so that if $i>i_0(R)$, then $\Sigma_i\cap \partial B_R\neq \emptyset$. Indeed, if this was not the case, then the mean curvature flows $\set{\sqrt{-t}\, \Sigma}_{t\in [-1,0)}$ and $\set{\partial B_{\sqrt{R^2-2n(t+1)}}}_{t\in [-1,0)}$ would violate the avoidance principle.
  
Now, let $\mu_i=\mu_{\Sigma_i}\in \mathcal{SM}_n[\Lambda-\epsilon_0]$. By the integral compactness theorem for $F$-stationary varifolds, up to passing to a subsequence the $\mu_i$ converge to a $\mu\in \mathcal{SM}_n[\Lambda-\epsilon_0]$.  By Proposition \ref{CondRegProp}, $\mu=\mu_\Sigma$ for some $\Sigma\in \mathcal{S}_n [\Lambda-\epsilon_0]$. Furthermore, up to passing to a further subspace, $\Sigma_i\to \Sigma$ in $C^\infty_{loc}(\Real^{n+1})$. It follows that $\Sigma \cap \partial B_R\neq \emptyset$ for all $R>\sqrt{2n}$.  In other words, $\Sigma$ is non-compact and so, by Proposition \ref{CondTotRegProp}, $\Sigma\in \mathcal{ACS}_n$. However, this implies that $\Sigma$ is non-collapsed (cf. \cite[Definition 4.6]{BernsteinWang}), while the $\Sigma_i$ are collapsed by \cite[Lemma 4.8]{BernsteinWang}. This contradicts \cite[Proposition 4.10]{BernsteinWang} and completes the proof.
\end{proof}

\section{Singularities of Flows with Small Entropy}\label{SingularitySec}
Given a Brakke flow $\mathcal{K}=\set{\mu_t}_{t\in I}$ and a point $(\xX_0,t_0)\in \sing(\mathcal{K})$ with $t_0\in \mathring{I}$, a tangent flow $\mathcal{T}\in\mathrm{Tan}_{(\xX_0,t_0)}\mathcal{K}$ is of \emph{compact type} if $\mathcal{T}=\set{\nu_t}_{t\in (-\infty,\infty)}$ and $\spt(\nu_{-1})$ is compact. Otherwise, the tangent flow is of \emph{non-compact type}. If every element of $\mathrm{Tan}_{(\xX_0, t_0)}\mathcal{K}$ is of compact type, then $(\xX_0,t_0)$ is a \emph{compact singularity}. Likewise, if every element of $\mathrm{Tan}_{(\xX_0, t_0)}\mathcal{K}$ is of non-compact type, then $(\xX_0,t_0)$ is a \emph{non-compact singularity}. 

For the remainder of this section, we fix a dimension $n\geq 3$ and constants $\Lambda\in (\lambda_n, \lambda_{n-1}]$\footnote{The reader may refer to Remark \ref{LowBndLam} for the reason that we restrict to $\Lambda>\lambda_n$.} and $\epsilon_0>0$, and suppose that both \eqref{Assump1} and \eqref{Assump2} hold. We further assume that $\Sigma_0\subset \Real^{n+1}$ is a closed connected hypersurface with $\lambda[\Sigma_0]\leq\Lambda-\epsilon_0$ and with the property that the level set flow $L(\Sigma_0)$ is non-fattening and that $(E,\mathcal{K})$ is the pair given by Theorem \ref{UnitdensityThm}.  

\begin{prop} \label{RegTanProp}
Let $(\xX_0,t_0)\in \sing(\mathcal{K})$ and $\mathcal{T}\in \mathrm{Tan}_{(\xX_0,t_0)}\mathcal{K}$. If $\mathcal{T}=\set{\nu_t}_{t\in (-\infty,\infty)}$ is of non-compact type, then $\nu_{-1}=\mu_\Sigma$ for some $\Sigma\in\mathcal{ACS}_n$. Moreover, there is a constant $R_1=R_1(n, \Lambda, \epsilon_0)$ so that for all $R\geq R_1$, 
$$\mathcal{T} \lfloor \left(B_{16R}\setminus\bar{B}_{R}\right)\times (-1,1)$$
is a smooth mean curvature flow.  Moreover, for all $\rho\in (R,16R)$ and $t\in (-1,1)$, $\partial B_\rho$ meets $\spt(\nu_t)$ transversally and $\partial B_\rho\cap \spt(\nu_t)$ is connected.
\end{prop}
\begin{proof}
First, invoking Theorem \ref{BlowupsThm} and the monotonicity formula, $\mathcal{T}$ is backwardly self-similar with respect to parabolic scalings about $(\mathbf{0},0)$ and $\nu_{-1}\in\mathcal{SM}_n [\Lambda-\epsilon_0]$. Furthermore, by Proposition \ref{CondTotRegProp}, we have $\nu_{-1}=\mu_\Sigma$ for some $\Sigma\in\mathcal{ACS}_n[\Lambda-\epsilon_0]$. Finally, by Corollary \ref{GraphCor}, the pseudo-locality property of mean curvature flow \cite[Theorem 1.5]{IlmanenNevesSchulze}\footnote{The proof of \cite[Theorem 1.5]{IlmanenNevesSchulze} uses the local regularity theorem of White, which is also applicable to the Brakke flows in Theorem \ref{UnitdensityThm} and their tangent flows -- see \cite[pp. 1487--1488]{WhiteReg}.} and Brakke's local regularity theorem, there is an $R_1>0$ depending only on $n,\Lambda,\epsilon_0$ so that for $R>R_1$,
$$\mathcal{T} \lfloor \left(B_{16R}\setminus \bar{B}_R\right)\times (-1,1)$$
is a smooth mean curvature flow. Indeed, for all $t\in (-1,1)$, $\spt(\nu_t)\cap \left(B_{16R}\setminus \bar{B}_R\right)$ is the graph of a function over a subset of $\mathcal{C}(\Sigma)$ the asymptotic cone of $\Sigma$ with small $C^2$ norm.  As such, for all $\rho\in (R,16R)$ and $t\in (-1,1)$, $\partial B_\rho$ meets $\spt(\nu_t)$ transversally. As $\lambda[\Sigma]\leq \lambda[\Sigma_0]<\lambda_{n-1}$ it follows from \cite[Theorem 1.1]{BernsteinWang2} that $\mathcal{L}(\Sigma)$, the link of $\mathcal{C}(\Sigma)$, is connected and, hence, so is $\partial B_\rho \cap \spt (\nu_t)$. 
\end{proof}

Next we observe that singularities are either compact or non-compact.
\begin{lem}\label{NonCpctOrCpctLem}
Each $(\xX_0,t_0)\in\sing(\mathcal{K})$ is either a compact or a non-compact singularity.
\end{lem}
\begin{proof}
Suppose that $(\xX_0,t_0)$ is not a non-compact singularity. Then there is a $\mathcal{T}=\{\nu_t\}_{t\in\Real}\in\mathrm{Tan}_{(\xX_0,t_0)}\mathcal{K}$ of compact type. By the monotonicity formula and Theorem \ref{BlowupsThm}, $\nu_{-1}\in\mathcal{SM}_n[\Lambda-\epsilon_0]$. 
It follows from Proposition \ref{CondTotRegProp} that $\nu_{-1}=\mu_\Sigma$ for some $\Sigma\in\mathcal{S}_n[\Lambda-\epsilon_0]$ and $\Sigma$ is closed. Hence, by \cite[Corollary 1.2]{Schulze}, $\mathcal{T}$ is the only element of $\mathrm{Tan}_{(\xX_0,t_0)}\mathcal{K}$ and so  $(\xX_0,t_0)$ is a compact singularity, proving the claim.
\end{proof}

We further prove that
\begin{thm}\label{IsolatedSingThm}
Given $(\xX_0, t_0)\in \sing(\mathcal{K})$, there exist $\rho_0=\rho_0(\xX_0,t_0, \mathcal{K})>0$ and $\alpha=\alpha(n,\Lambda,\epsilon_0)>1$ so that:
\begin{enumerate}
\item If $(\xX_0,t_0)$ is a compact singularity and $\rho<\rho_0$, then
$$\mathcal{K}\lfloor \left(B_{2\alpha \rho}(\xX_0)\times (t_0-4\alpha^2\rho^2, t_0+4\alpha^2\rho^2) \backslash \set{(\xX_0,t_0)} \right) $$
is a smooth mean curvature flow. Furthermore, for all $R\in (\frac{1}{2}\alpha \rho, 2\alpha \rho)$ and  $t\in (t_0 -\rho^2, t_0+\rho^2)$,  $\spt(\mu_t)\cap \partial B_{R}(\xX_0)=\emptyset$.
\item If $(\xX_0,t_0)$ is a non-compact singularity and $\rho<\rho_0$, then
$$\mathcal{K}\lfloor \left(B_{2\alpha \rho}(\xX_0)\times (t_0-4\alpha^2\rho^2, t_0] \backslash \set{(\xX_0,t_0)} \right) $$
and
$$\mathcal{K}\lfloor \left(B_{2\alpha\rho}(\xX_0)\backslash \bar{B}_{\frac{1}{2}\alpha\rho}(\xX_0)\right)\times (t_0-\rho^2, t_0+\rho^2)$$
are both smooth mean curvature flows.  Furthermore, for all  $R\in (\frac{1}{2}\alpha \rho, 2\alpha \rho)$ and $t\in (t_0 -\rho^2, t_0+\rho^2)$,  $\partial B_{R}(\xX_0)$ meets $\spt(\mu_t)$ transversally and the intersection is connected. 
\end{enumerate}
Finally, for all $t\in (t_0-\rho^2, t_0)$, $\spt(\mu_t)\cap \bar{B}_{\alpha \rho}(\xX_0)$ is diffeomorphic (possibly as a manifold with boundary) to $\Gamma\cap \bar{B}_{\alpha}$, where $\Gamma\in \mathcal{S}_n^*[\Lambda-\epsilon_0]$ and, if $\Gamma\in \mathcal{ACS}_n$, then $\Gamma\backslash B_\alpha$ is diffeomorphic to $\mathcal{L}(\Gamma)\times[0,\infty)$.
\end{thm}
\begin{proof}
Set $\alpha=4\max\set{R_1, R_D,1}$ where $R_1$ is given by Proposition \ref{RegTanProp} and $R_D$ is given by Proposition \ref{DiamBoundProp}.  Without loss of generality, we may assume that $(\xX_0,t_0)=(\mathbf{0},0)$.

We establish the regularity near (but not at) $(\OO,0)$ by contradiction.  To that end, suppose that there was a sequence of points  $(\xX_i,t_i)\in \sing(\mathcal{K})\backslash \set{(\OO, 0)}$ such that $(\xX_i,t_i)\to (\mathbf{0},0)$. If $(\mathbf{0},0)$ is a non-compact singularity, we further assume $t_i\leq 0$. Let $r_i^2=|\xX_i|^2+|t_i|$. Then, up to passing to a subsequence, it follows from Theorem \ref{BlowupsThm} that $\mathcal{K}^{(\mathbf{0},0),r_i}\to\mathcal{T}$ in the sense of Brakke flows and $\mathcal{T}=\{\nu_t\}_{t\in\Real}\in\mathrm{Tan}_{(\mathbf{0},0)}\mathcal{K}$. Let $\tilde{\xX}_i=r^{-1}_i \xX_i$ and $\tilde{t}_i=r^{-2}_i t_i$. Then $|\tilde{\xX}_i|^2+|\tilde{t}_i|=1$, that is, $(\tilde{\xX}_i,\tilde{t}_i)$ lies on the unit parabolic sphere in space-time. Thus, up to passing to a subsequence, $(\tilde{\xX}_i,\tilde{t}_i)\to (\tilde{\xX}_0,\tilde{t}_0)$, where $|\tilde{\xX}_0|^2+|\tilde{t}_0|=1$. Moreover, the upper semi-continuity of Gaussian density implies that $\Theta_{(\tilde{\xX}_0,\tilde{t}_0)}(\mathcal{T})\geq 1$.  
 
As $\nu_{-1}\in \mathcal{SM}_n [\Lambda-\epsilon_0]$, Proposition \ref{CondTotRegProp} implies that $\sing(\nu_t)=\emptyset$ for $t<0$. That is, $(\tilde{\xX}_0,\tilde{t}_0)$ is a regular point of $\mathcal{T}$ if $\tilde{t}_0<0$.
If $(\mathbf{0},0)$ is a non-compact singularity, then $\mathcal{T}$ is of non-compact type and $\tilde{t}_0\leq 0$.  Hence, either $(\tilde{\xX}_0,\tilde{t}_0)$ is a regular point or $\tilde{t}_0=0$ and $|\tilde{\xX}_0|=1$.  However in the later case, Proposition \ref{RegTanProp} applied to $ \mathcal{T}^{(\mathbf{0}, 0), \alpha}\in\mathrm{Tan}_{(\mathbf{0},0)}\mathcal{K}$ implies that $(\tilde{\xX}_0,\tilde{t}_0)$ is also a regular point of $\mathcal{T}$. 
If $(\mathbf{0},0)$ is a compact singularity, then $\mathcal{T}$ is of compact type and $\nu_{-1}=\mu_\Gamma$ for some $\Gamma\in\mathcal{S}_n(\Lambda)$ by Proposition \ref{CondTotRegProp}. This implies that $\mathcal{T}$ is extinct at time $0$ and $\sing(\mathcal{T})=\{(\mathbf{0},0)\}$, again implying that $\tilde{t}_0\leq 0$ and $(\tilde{\xX}_0,\tilde{t}_0)$ is a regular point of $\mathcal{T}$.
Hence, it follows from Brakke's local regularity theorem that for all $i$ sufficiently large, $(\tilde{\xX}_i,\tilde{t}_i)\notin\sing(\mathcal{K}^{(\mathbf{0},0),r_i})$, or equivalently, $(\xX_i,t_i)\notin\sing(\mathcal{K})$. This is the desired contradiction. Therefore, for $\rho_0'>0$ sufficiently small, if $\rho<\rho_0'$  and $(\mathbf{0},0)$ is a non-compact singularity, then
$$
\mathcal{K} \lfloor \left(B_{2\alpha\rho}\times (-4\alpha^2\rho^2,0]\setminus\{(\mathbf{0},0)\}\right)
$$
is a smooth mean curvature flow, while, if $\rho<\rho_0'$ and $(\mathbf{0},0)$ is a compact singularity, then
$$
\mathcal{K} \lfloor \left(B_{2\alpha\rho}\times (-4\alpha^2 \rho^2,4\alpha^2 \rho^2)\setminus\{(\mathbf{0},0)\}\right)
$$
is a smooth mean curvature flow.

We continue arguing by contradiction and again consider a sequence, $\rho_i$, of positive numbers with $\rho_i\to 0$ and $\rho_i<\rho_0'$. Up to passing to a subsequence, $\mathcal{K}^{(\mathbf{0},0),\rho_i}$ converges, in the sense of Brakke flows, to some $\mathcal{T}=\{\nu_t\}_{t\in\Real}\in\mathrm{Tan}_{(\mathbf{0},0)}\mathcal{K}$.  If $(\OO,0)$ is a compact singularity, then, as $\alpha\geq 4 R_D$, $\partial B_R \cap \spt(\nu_{t})=\emptyset$ for $R\geq\frac{1}{2}\alpha$ and $t\in (-1,1)$ by Proposition \ref{DiamBoundProp} and the avoidance principle. Hence, the nature of the convergence implies that, for $\rho_i$ sufficiently large, $\partial B_{R}\cap \spt(\mu_t)=\emptyset$ for $t\in (-\rho^2_i,\rho^2_i)$ and $R\in (\frac{1}{2} \alpha \rho_i, 2\alpha \rho_i)$.  
If $(\OO, 0)$ is a non-compact singularity, then Proposition \ref{RegTanProp},  implies that
$$
\mathcal{T} \lfloor \left(B_{4\alpha}\setminus \bar{B}_{\frac{1}{4}\alpha}\right)\times (-1,1)
$$
is a smooth mean curvature flow and for all $R\in (\frac{1}{4}\alpha,4\alpha)$ and $t\in (-1,1)$, $\partial B_R$ meets $\spt(\nu_t)$ transversally and as a connected set. Thus, by Brakke's local regularity theorem, for all $i$ sufficiently large,
$$
\mathcal{K}^{(\mathbf{0},0),\rho_i} \lfloor \left(B_{2\alpha}\setminus\bar{B}_{\frac{1}{2} \alpha}\right)\times (-1,1)
$$
is a smooth mean curvature flow, and hence so is
$$
\mathcal{K} \lfloor \left(B_{2\alpha \rho_i}\setminus\bar{B}_{\frac{1}{2}\alpha \rho_i}\right) \times (-\rho_i^2,\rho_i^2).
$$
Moreover, for all $R\in (\frac{1}{2}\alpha \rho_i,2\alpha\rho_i )$ and $t\in (-\rho_i^2,\rho_i^2)$, $\partial B_R$ meets $\mu_t$ transversally and as a connected set.
Hence, as the sequence $\rho_i$ was arbitrary, there is a $\rho_0''<\rho_0'$ so that Items (1) and (2) hold for $\rho<\rho_0''$.  

To complete the proof, we observe that again arguing by contradiction, there is a $\rho_0<\rho_0''$ 
so that if $\rho<\rho_0$,  $B_{2\alpha }\cap \rho^{-1} \spt(\mu_{-\rho^2})$ is a normal graph over a domain $\Omega$ in $\Gamma$ with small $C^2$ norm for some $\Gamma\in \mathcal{S}_n [\Lambda-\epsilon_0]$. In particular, by Corollary \ref{GraphCor}, $\partial\Omega$ is a small normal graph over $\partial B_{\alpha} \cap \Gamma$, so $\bar{B}_{\alpha\rho }\cap \spt(\mu_{-\rho^2})$ is diffeomorphic to $\bar{B}_\alpha \cap \Gamma$.  Furthermore, the choice of $\alpha$ ensures that if $\Gamma\in \mathcal{ACS}_n$, then $\Gamma\backslash B_{\alpha}$ is diffeomorphic to $\mathcal{L}(\Sigma)\times [0, \infty)$.
It remains only to show that $\bar{B}_{\alpha\rho }\cap \spt(\mu_{t})$ is diffeomorphic to $\bar{B}_\alpha \cap \Gamma$ for $t\in (-\rho^2, 0)$. This follows from the fact that, as already established, the flow is smooth in $\bar{B}_{2\alpha \rho} \times [-2\rho^2, 0)$ and, for all $t\in [-\rho^2,0)$, either $\partial B_{\alpha\rho} \cap \spt(\mu_t)=\emptyset$ (if the singularity is compact) or the intersection is transverse (if the singularity is non-compact).  As such, the flow provides a diffeomorphism between $\bar{B}_{\alpha\rho }\cap \spt(\mu_{t})$ and $\bar{B}_{\alpha\rho }\cap \spt(\mu_{-\rho^2})$ -- see Appendix A.
\end{proof}

We obtain a direct consequence of Theorem \ref{IsolatedSingThm}.
\begin{cor}\label{TimeSingCor}
For each $t_0>0$, $\sing_{t_0}(\mathcal{K})=\set{\xX: (\xX,t_0)\in \sing(\mathcal{K})}$ is finite.
\end{cor}

Given a manifold $M$ we say a subset $U\subset M$ is a \emph{smooth domain} if $U$ is open and $\partial U$ is a smooth submanifold. 
\begin{thm}\label{CondSurgThm}
There is an $N=N(\Sigma_0)\in\mathbb{N}$ and a sequence of closed connected hypersurfaces $\Sigma^1, \ldots, \Sigma^N$ so that:
	\begin{enumerate}
		\item $\Sigma^1=\Sigma_0$;
		\item $\Sigma^N$ is diffeomorphic to $\mathbb{S}^n$;
		\item For each $i$ with $1\leq i \leq N-1$, there is an $m=m(i)\in \mathbb{N}$ and open connected pairwise disjoint smooth domains $U_1^i, \ldots, U_{m(i)}^i \subset \Sigma^i$ and $V_1^i, \ldots, V_{m(i)}^i \subset \Sigma^{i+1}$ so that:
		\begin{itemize}
			\item There are orientation preserving diffeomorphisms 
			$$\hat{\Phi}^{i}:\Sigma^{i+1}\backslash \cup_{j=1}^{m(i)} V_j^{i}\to \Sigma^{i}\backslash \cup_{j=1}^{m(i)} U_j^i;$$
			\item Each $\bar{U}_j^i$ is diffeomorphic to $\bar{B}_{R_j^i}\cap \Gamma_j^i$ where $\Gamma_j^i\in \mathcal{ACS}_n^*(\Lambda)$ and $\Gamma_j^i\backslash B_{R_j^i}$ is diffeomorphic to $\mathcal{L}(\Gamma_j^i)\times [0,\infty)$.
		\end{itemize}
	\end{enumerate}
\end{thm}
\begin{proof}
	
 Let us denote the set of compact singularities of $\mathcal{K}$ by $\sing^C(\mathcal{K})$ and the set of non-compact singularities by $\sing^{NC}(\mathcal{K})$.  By Lemma \ref{NonCpctOrCpctLem}, $\sing(\mathcal{K})=\sing^{NC}(\mathcal{K})\cup \sing^C(\mathcal{K})$. We note that if $X\in \sing^{NC}(\mathcal{K})$, then, by Proposition \ref{CondTotRegProp}, every element of $\mathrm{Tan}_X \mathcal{K}$ is the flow of an element of $\mathcal{ACS}_n$ and so the tangent flows are non-collapsed at time $0$ in the sense of \cite[Definition 4.9]{BernsteinWang}.  Hence, by \cite[Lemma 5.1]{BernsteinWang}, $\sing^C(\mathcal{K})\neq \emptyset$. In fact, if we define the extinction time of $\mathcal{K}$ to be 
$$
T(\mathcal{K})=\sup\set{t: \spt(\mu_t)\neq \emptyset},
$$
then
$$
\emptyset \neq \set{\xX\in \Real^{n+1}: \Theta_{(\xX_0,T(\mathcal{K}))} (\mathcal{K})\geq 1}=\set{\xX\in \Real^{n+1}: (\xX, T(\mathcal{K}))\in \sing{}^C (\mathcal{K})}.
$$
It follows from Theorem \ref{IsolatedSingThm} that $\sing^C(\mathcal{K})$ consists of at most a finite number of points.  
	
Observe that if $\sing(\mathcal{K})$ consists of exactly one point $X_0$, then we can take $N=1$. Indeed, by the above discussion, this singularity must be compact and hence, by Proposition \ref{CondTotRegProp}, there is a $\Gamma\in \mathcal{S}_n(\Lambda)$ diffeomorphic to $\mathbb{S}^n$ so that one of the tangent flows at $X_0$ is the flow associated to $\mu_{\Gamma}$. In this case we may write $\mathcal{K}=\set{\mu_{\Sigma_t}}_{t\in [0, T(\mathcal{K}))}$ where $\set{\Sigma_t}_{t\in [0, T(\mathcal{K}))}$ is a smooth mean curvature flow. By Brakke's regularity theorem, there is a $t$ near $T(\mathcal{K})$ so that $\Sigma_t$ is a small normal graph over $\Gamma$ and hence $\Sigma^1=\Sigma_0$ is diffeomorphic to $\Gamma$, verifying the claim.
	
Now let $\mathrm{ST}(\mathcal{K})=\set{t\in \Real: (\xX,t)\in \sing(\mathcal{K})}$ be the set of singular times.  Notice that by Corollary \ref{TimeSingCor} there are at most a finite number of singular points associated to each singular time. We observe that as $\Sigma^1=\Sigma_0$ is smooth, there is a $\delta>0$ so that $\mathrm{ST}(\mathcal{K})\subset [\delta, T(\mathcal{K})]$. Furthermore, as $\sing(\mathcal{K})$ is a closed set, so is $\mathrm{ST}(\mathcal{K})$.
	
For each $t\in \mathrm{ST}(\mathcal{K})$, let 
$$
\rho(t)=\min\set{\rho_0(\xX, t, \mathcal{K}): \xX\in \sing{}_t(\mathcal{K})}>0,
$$
where $\rho_0(\xX, t, \mathcal{K})$ is the constant given by Theorem \ref{IsolatedSingThm}. This minimum is positive as $\sing_{t}(\mathcal{K})$ is a finite set. Observe that by Theorem \ref{IsolatedSingThm}, 
\begin{equation} \label{DisjointEqn}
B_{\alpha\rho(t)}(\xX)\cap B_{\alpha\rho(t)}(\xX^\prime)=\emptyset
\end{equation} 
when $\xX,\xX^\prime$ are distinct elements of $\sing_t(\mathcal{K})$ and $\alpha=\alpha(n, \Lambda, \epsilon_0)$ is given by Theorem \ref{IsolatedSingThm}. Next, choose $\tau(t)\in (0,\rho^2(t))$ so that 
$$
\mathcal{K} \lfloor \left(\Real^{n+1}\setminus\bigcup_{\xX\in\sing_t(\mathcal{K})} \bar{B}_{\alpha\rho(t)}(\xX)\right)\times \left(t-\tau(t),t+\tau(t)\right)
$$
is a smooth mean curvature flow.  Such a $\tau$ exists as $\sing(\mathcal{K})$ is a closed set.

As $\mathrm{ST}(\mathcal{K})$ is a closed subset of $[0,T(\mathcal{K})]$, it is a compact set and so the open cover
$$
\set{(t-\tau(t), t+\tau(t)): t\in \mathrm{ST}(\mathcal{K})}
$$
of $\mathrm{ST}(\mathcal{K})$ has a finite subcover. That is, there are a finite number of times $t_1, \ldots, t_{N^\prime}\in \mathrm{ST}(\mathcal{K})$, labeled so that $t_i<t_{i+1}$ and chosen so that 
$$
\mathrm{ST}(\mathcal{K})\subset \bigcup_{i=1}^{N^\prime} (t_i-\tau(t_i),t_i+\tau(t_i)).
$$
Furthermore, we can assume that for each $i$:
\begin{enumerate}
\item For all $j>i$, $t_i-\tau(t_i)< t_j- \tau(t_j)$,
\item For all $j<i$, $t_i+\tau(t_i)>t_j+\tau(t_j)$, and
\item For all $j<i<j'$, $t_j+\tau(t_j)<t_{j'}-\tau(t_{j'})$.
\end{enumerate} 
As otherwise, we could delete $(t_i-\tau(t_i),t_i+\tau(t_i))$ and still have an open cover. 
Note that, by the definition of $\tau(t)$, one must have $t_{N^\prime}=T(\mathcal{K})$.  

By Theorem \ref{IsolatedSingThm} we may choose a sequence of points $s^\pm_1, \ldots, s_{N^\prime}^\pm$ with $t_i\in (s_i^-, s_i^+)$, $|s_i^\pm-t_i|<\tau(t_i)$,  $s_i^+\leq s_{i+1}^-$ and so that 
$$
\left([0, s_1^-]\cup \bigcup_{i=1}^{N^\prime-1}[s_i^+, s_{i+1}^-]\right)\cap \mathrm{ST}(\mathcal{K})=\emptyset.
$$
More concretely, first take $s_1^-\in (t_1-\tau(t_1), t_1)$ with $s_1^->0$ and $s_{N^\prime}^+=t_{N^\prime}+\frac{1}{2}\tau(t_{N^\prime})$. 
 For $1\leq i\leq N^\prime-1$, let $$\tilde{s}_i^+=\sup \left( \mathrm{ST}(\mathcal{K})\cap (t_i-\tau(t_i), t_i+\tau(t_i))\right)$$
 and for $2\leq i \leq N^\prime$, let 
 $$\tilde{s}_i^-=\inf \left( \mathrm{ST}(\mathcal{K})\cap (t_i-\tau(t_i), t_i+\tau(t_i))\right).$$  The definition of $\tau(t_i)$ and Theorem \ref{IsolatedSingThm} imply that $\tilde{s}_i^-=t_i$. 
 As the set of singular times is closed and $t_i\in \mathrm{ST}(\mathcal{K})$,  $\tilde{s}_i^+\in \mathrm{ST}(\mathcal{K})$ and $t_i \leq \tilde{s}^+_i$.  
 We treat two cases.  In the first case we suppose that $t_{i+1}-\tau(t_{i+1})<t_i+\tau(t_i) $.  As $\tilde{s}_{i+1}^-=t_{i+1}$, there are then no singular times in the interval $(t_{i+1}-\tau(t_{i+1}),t_i+\tau(t_i))$ and so we may take $s_i^+=s_{i+1}^-$ to be the same point in this interval.  In the second case, we suppose that $t_i+\tau(t_i)\leq t_{i+1}-\tau(t_{i+1})$ and observe that $\tilde{s}_i^+\leq t_i+\tau(t_i)\leq t_{i+1}-\tau(t_{i+1})$.  In fact, $\tilde{s}_i^+<t_i+\tau(t_i)$ as otherwise in order to cover $\mathrm{ST}(\mathcal{K})$ assumption (3) from above would not hold.  
   Pick $s_i^+$ as some point in $(\tilde{s}_i^+, t_i+\tau(t_i))$ and $s_{i+1}^-$ as some point in $( t_{i+1}-\tau(t_{i+1}), t_i)$.  The lack of singular times in $[0,s_0^-]$ and in each $[s_i^+, s_{i+1}^-]$ follows by our choices and assumptions (1) and (3) above. 
   
For $1\leq i \leq N^\prime$ set $\Sigma^i_\pm = \spt(\mu_{s_i^\pm})$. By the choice of $s_i^\pm$, each $\Sigma^i_\pm$ is a closed hypersurface and, as there are no singular times between $s_i^+$ and $s_{i+1}^-$, we have for $1\leq i \leq N^\prime-1$ diffeomorphisms $\Phi^i: \Sigma_+^i \to \Sigma_-^{i+1}$ coming from the flow and, for the same reason, a diffeomorphism $\Phi^0: \Sigma^1\to \Sigma^1_-$.  Observe that, \emph{a priori}, the $\Sigma^i_\pm$ need not consist of one component (indeed, $\Sigma^{N^\prime}_+$ is empty).  By Corollary \ref{TimeSingCor}, $\sing_{t_i}(\mathcal{K})$ is finite for each $1\leq i \leq N^\prime$ and we write
$$
\set{\xX_i^1, \ldots, \xX_i^{M(i)}}=\sing{}_{t_i}(\mathcal{K})
$$
i.e., the $(\xX_i^j,t_i)$ are the singular points of the flow at time $t_i$.  Up to relabeling, there is an $0\leq m(i)\leq M(i)$ so that for $1\leq j \leq m(i)$, $(\xX_i^j, t_i)\in \sing^{NC}(\mathcal{K})$ while for $m(i)<j\leq M(i)$, $(\xX_i^j,t_i)\in \sing^C(\mathcal{K})$.  Set $R^i=\alpha\rho(t_i)$ and, for each $\xX_i^j$, let $U^{i}_{j,\pm}\subset \Sigma^i_\pm$ be the sets $ B_{R^i}(\xX_i^j)\cap \Sigma^i_\pm$. By \eqref{DisjointEqn} for fixed $j$, these are pairwise disjoint sets and, by Theorem \ref{IsolatedSingThm}, these intersections are transverse and so the $\sigma^{i}_{j,\pm}=\partial U^i_{j,\pm}$ are submanifolds of $\Sigma^i_\pm$.  Hence,  the $U^i_{j,\pm}$ are smooth pairwise disjoint domains.

 Furthermore, by Theorem \ref{IsolatedSingThm} and fact that $\tau(t)<\rho(t)$, each $\bar{U}^i_{j,-}$ is diffeomorphic to $\bar{B}_{\alpha}\cap \Gamma_{j}^i$ for some $\Gamma_{j}^i \in \mathcal{S}_n$.  In particular, for $j>m(i)$ we have that $\bar{U}^i_{j,-}$ is a closed connected hypersurface, while for $1\leq j\leq m(i)$, $\partial \bar{U}^i_{j,-}$ is non-empty and connected. Hence, for $j>m(i)$, $\bar{U}^i_{j,+}=\emptyset$, while for $1\leq j\leq m(i)$, $\partial \bar{U}^i_{j,+}$ is non-empty and connected.
Furthermore, Theorem \ref{IsolatedSingThm} implies that there are diffeomorphisms (see Appendix A)
$$\Psi^i: \Sigma_-^i\backslash \bigcup_{j=1}^{M(i)} U_{j,-}^i \to \Sigma_+^i \backslash \bigcup_{j=1}^{M(i)} U_{j,+}^i.$$

As $\Sigma^1$ is connected and $\Phi^0(\Sigma^1)=\Sigma^1_-$, $\Sigma^1_-$ is also connected. As each $\sigma^1_{j,-}$ is connected, we obtain that $\hat{\Sigma}^1_-=\Sigma^1_-\backslash \bigcup_{j=1}^{M(1)} U_{j,-}^1$ is connected. Let $\tilde{\Sigma}^1_+$ be the connected component of $\Sigma^1_+$ that contains $\Psi^1(\hat{\Sigma}^1_{-})$.  Inductively, let $\tilde{\Sigma}^{i+1}_-=\Phi^i(\tilde{\Sigma}^i_+)$ and $\hat{\Sigma}^{i+1}_-=\tilde{\Sigma}^{i+1}_-\backslash \bigcup_{j=1}^{M(i+1)} U_{j,-}^{i+1}$ and define $\tilde{\Sigma}^{i+1}_+$ to be the connected component of $\Sigma^{i+1}_+$ that contains $\Psi^{i+1}(\hat{\Sigma}^{i+1}_-)$. Here we adopt the convention that if $\hat{\Sigma}^{i+1}_-=\emptyset$, then $\tilde{\Sigma}^{i+1}_+=\emptyset$. It follows inductively  that each $\tilde{\Sigma}^i_{\pm}$ is connected. 
Let $\tilde{\Phi}^i: \tilde{\Sigma}^i_+\to \tilde{\Sigma}^{i+1}_-$ be the diffeomorphisms given by restricting the $\Phi^i$.  To be consistent we also set $\tilde{\Sigma}^1_-=\Sigma^1_-$ and $\tilde{\Phi}^0=\Phi^0$. 

Finally let
$$
N=\max\set{1\leq i\leq N^\prime: \tilde{\Sigma}^k_-\neq\emptyset\mbox{ for all $1\leq k\leq i$}}.
$$
If $N<N^\prime$, then, by constructions, $\hat{\Sigma}^N_-=\emptyset$ and $\tilde{\Sigma}^N_- = U^N_{j,-}$ for some $j>m(N)$. If $N=N^\prime$, then $t_N=T(\mathcal{K})$ at which all singularities are compact. Thus it follows from \cite[Theorem 0.7]{CIMW} that $\tilde{\Sigma}^N_-$ is diffeomorphic to $\mathbb{S}^n$. The theorem now follows by taking $\Sigma^i=\tilde{\Sigma}^i_-$ for $2\leq i \leq N$ and $\hat{\Phi}^i$ are the diffeomorphisms given by $(\tilde{\Phi}^i\circ\Psi^i)^{-1}$.
\end{proof}
%

\section{A Sharpening of \cite{BernsteinWang2}}\label{SharpeningSec}
In order to prove Theorem \ref{MainACSThm}, we begin with an elementary lemma.
\begin{lem}\label{StupidLem}
	If $\xX_1, \ldots, \xX_{m+1}\in \Real^{n+1}$ is a sequence of points so that
	\begin{equation} \label{StupidHyp}
	|\xX_i-\xX_{i+1}|\leq \hat{K} (1+|\xX_i|)^{-1}
	\end{equation}
	for  $1\leq i \leq m$ and some $\hat{K}\geq 0$, then
	\begin{equation} \label{StupidClaim}
	|\xX_1-\xX_{m+1}|\leq K(m)  (1+|\xX_1|)^{-1}
	\end{equation}
	where $K(m)=(\hat{K}+1)^m-1$.
\end{lem}
\begin{proof}
	We proceed by induction on $m$. The lemma is obviously true when $m=1$. Suppose \eqref{StupidClaim} holds for $m=m'$.  Using this induction hypothesis with \eqref{StupidHyp} implies that
	$$
	|\xX_1-\xX_{m'+2}|\leq |\xX_1-\xX_{m'+1}|+|\xX_{m'+1}-\xX_{m'+2}|\leq K(m') (1+|\xX_1|)^{-1}+\hat{K} (1+|\xX_{m'+1}|)^{-1}.
	$$
	Furthermore, by the induction hypothesis and triangle inequality
	$$ |\xX_{1}|\leq K(m') (1+|\xX_{1}|)^{-1}+ |\xX_{m'+1}|.$$
	As $K(m')\geq 0$ and $(1+|\xX_{1}|)^{-1}\leq 1$, this implies that
	$$
	1+ |\xX_{1}| \leq 1+K(m')+|\xX_{m'+1}|\leq (1+K(m'))  (1+|\xX_{m'+1}|).
	$$
	That is,
	$$ (1+|\xX_{m'+1}|)^{-1}\leq (1+K(m')) (1+|\xX_{1}|)^{-1}.
	$$
	Hence, 
	$$
	|\xX_1-\xX_{m'+2}|\leq (K(m')+\hat{K}(1+K(m'))) (1+|\xX_1|)^{-1}
	$$
	and, by the induction hypothesis, $K(m')=(\hat{K}+1)^{m'}-1$ and so setting
	$$
	K(m'+1)=K(m')+\hat{K}(1+K(m'))=(\hat{K}+1)^{m'+1}-1
	$$
    verifies that \eqref{StupidClaim} holds for $m=m'+1$ and finishes the proof.
\end{proof}

We next observe that the proof of the main result of \cite[Theorem 0.1]{BernsteinWang2} actually allows us to make the following more refined conclusion.
\begin{prop} \label{MainACSProp}
Fix $n\geq 2$, if $\Sigma\in \mathcal{ACS}_n[\lambda_{n-1}]$, then there is a homeomorphic involution $\phi:\mathbb{S}^{n}\to \mathbb{S}^{n}$  which fixes $\mathcal{L} (\Sigma)$, the link of the asymptotic cone, $\mathcal{C} (\Sigma)$, of $\Sigma$, and swaps the two components of $\mathbb{S}^n\backslash \mathcal{L}(\Sigma)$.  
\end{prop}
\begin{proof}
By \cite[Theorem 0.1]{BernsteinWang2}, the link $\mathcal{L}(\Sigma)$ is connected and separates $\mathbb{S}^n$ into two components $\Omega_+$ and $\Omega_-$.  In particular,  $\mathcal{L}(\Sigma)=\partial \bar{\Omega}_+=\partial \bar{\Omega}_-$.  In order to construct $\phi$, it is enough to show the existence of a homeomorphism ${\psi}: \bar{\Omega}_+\to \bar{\Omega}_-$ so that ${\psi}|_{\mathcal{L}(\Sigma)} : \mathcal{L}(\Sigma)\to \mathcal{L}(\Sigma)$ is the identity map.
Indeed, if such a ${\psi}$ exists, one defines $\phi$ by
$$
\phi(p)=\left\{ \begin{array}{ll} {\psi}(p) & p\in \bar{\Omega}_+ \\ {\psi}^{-1}(p) & p \in \Omega_- \end{array}\right.
$$

To explain the construction of ${\psi}$ let us first summarize the main objects used in the proof of \cite[Theorem 0.1]{BernsteinWang2}.   First, recall that it is shown there that associated to $\Sigma$ are two smooth mean curvature flows $\set{\Gamma_t^{\pm }}_{t\in[-1,0]}$ with $\Gamma_{-1}^{+}$ the normal exponential graph over $\Sigma$ of a small positive multiple of the lowest eigenfunction of the self-shrinker stability operator of $\Sigma$ (normalized to be positive) and $\Gamma^-_{-1}$ to be a small negative multiple of this function.  In particular,  by choosing the multiple small enough, one can ensure both that
$\Gamma^+_{-1}$ is the exponential normal graph of some function on $\Gamma^-_{-1}$ and that $\Gamma^-_{-1}$ is the exponential normal graph of some function on $\Gamma^+_{-1}$.
 Furthermore, up to relabeling, each $\Gamma^{\pm}=\Gamma^{\pm}_0$ is diffeomorphic to $\Omega^\pm$ the components of $\mathbb{S}^n\backslash \mathcal{L}(\Sigma)$.  Moreover, these diffeomorphisms, which we denote by $\Pi^\pm$, are given by restricting the map 
\begin{equation*}
\Pi(p)=\frac{\xX(p)}{|\xX(p)|}
\end{equation*}
to $\Gamma^\pm$.

We next use the flow $\set{\Gamma^\pm_{t}}_{t\in[-1,0]}$ to construct a natural diffeomorphism ${\Psi}: \Gamma^+\to \Gamma^-$ which has the property that there is a constant $K>0$ so that
\begin{equation}
\label{DistortionEst}\left|\xX(p) -\xX({\Psi}(p))\right|\leq \frac{K}{1+|\xX(p)|}.
\end{equation}
We do so iteratively.  Specifically, by \cite[Items (1) and (2) of Proposition 4.4 and Proposition 4.5]{BernsteinWang2} there is a constant $\tilde{C}_0>0$ so that
\begin{equation}
\label{GammaCurvEst}
\sup_{t\in [-1,0]} \sup_{\Gamma_t^\pm} \left( |A_{\Gamma^\pm_t}|+|\nabla_{\Gamma^\pm_t} A_{\Gamma^\pm_t}|\right)<\tilde{C}_0.
\end{equation}
This, together with \cite[Item (3) of Proposition 4.4]{BernsteinWang2}, implies that there is a $\rho>0$ so that for each $t\in [-1,0]$, 
$\mathcal{T}_{\rho}(\Gamma_t^\pm)$ is a regular tubular neighborhood of $\Gamma_t^\pm$.  It follows from this and \eqref{GammaCurvEst} that there is a $\delta>0$ so that if $t_1, t_2\in [-1,0]$ and $|t_1-t_2|<\delta$, then $\Gamma^\pm_{t_1}$ is a normal exponential graph over $\Gamma^\pm_{t_2}$ and vice versa.  As such, for all $t_1, t_2\in [-1,0]$ with $|t_1-t_2|<\delta$, there is a diffeomorphism  
$${\Psi}^\pm_{t_2,t_1}: \Gamma_{t_1}^\pm \to \Gamma_{t_2}^\pm$$
defined by nearest point projection from $\Gamma_{t_1}^\pm$ to $\Gamma^{\pm}_{t_2}$.
Pick $M\in \mathbb{N}$ so $M\delta>1$ and choose $0=s_0>s_1>\ldots >s_M=-1$ so that $|s_i-s_{i+1}|<\delta$ and define a diffeomorphism $\Psi^-: \Gamma^-_{-1}\to \Gamma^-$ by 
$$
\Psi^-= \Psi^-_{s_0, s_1}\circ \Psi^-_{s_1,s_2} \circ \cdots \circ \Psi^-_{s_{M-1}, s_{M}}.
$$
Likewise, define a diffeomorphism $\Psi^+:\Gamma^+ \to \Gamma^+_{-1}$ by
$$
\Psi^+=\Psi^+_{s_{M}, s_{M-1}}\circ \Psi^+_{s_{M-1}, s_{M-2}}\circ \cdots \circ \Psi^+_{s_1,s_0}
$$
 and let $\Psi^{+,-}:\Gamma^+_{-1}\to \Gamma^-_{-1}$ be given by nearest point projection.  By construction, this is also a diffeomorphism and so the map
$$
\Psi=\Psi^-\circ \Psi^{+,-}\circ \Psi^+
$$
is a diffeomorphism $\Psi:\Gamma^+\to \Gamma^-$.  

By construction, if $t_1,t_2\in [-1,0]$ and $|t_1-t_2|<\delta$, then for all $p\in \Gamma_{t_1}^\pm$,
\begin{equation}
\label{SillyEst}
|\xX(p)-\xX(\Psi_{t_2, t_1}^\pm(p) )|<\rho.
\end{equation}
Furthermore, \cite[Item (1) of Proposition 4.4]{BernsteinWang2} implies that for $t\in [-1,0]$ each $\Gamma_t^\pm $ is smoothly asymptotic to $\mathcal{C}(\Sigma)$. In particular, there is a $R>0$ and functions $u_t^\pm$ on $\mathcal{C}(\Sigma)\backslash B_R$ whose normal exponential graph over $\mathcal{C}(\Sigma)$ sits inside of $\Gamma^\pm_t$ and contains $\Gamma^\pm_t\backslash B_{2R}.$
Moreover,  by \cite[Item (2) of Proposition 4.2]{BernsteinWang2} and \cite[Lemma 4.3]{BernsteinWang2} 
there is a constant $K^\prime>0$ so that for $p\in \mathcal{C}(\Sigma)\backslash B_R$,
$$
|u_t^\pm(p)|\leq K^\prime (1+|\xX(p)|)^{-1}.
$$

Hence, for any $t_1,t_2\in [-1,0]$, if $p \in \Gamma^\pm_{t_1}\backslash B_{2R}$, then there is a point $p'\in \mathcal{C}(\Sigma)\backslash B_R$ so that
\begin{equation}
\label{DecayEst}
|\xX(p)-\xX(p')|\leq K^\prime (1+ |\xX(p')|)^{-1}
\end{equation}
and also a point $p''\in \Gamma^\pm_{t_2}$ so that
\begin{equation}
|\xX(p')-\xX(p'')|\leq K^\prime (1+ |\xX(p')|)^{-1}.
\end{equation}
Hence, if $|t_1-t_2|<\delta$, then as $\Psi^\pm_{t_2,t_1}$ is given by nearest point projection,
\begin{align*}
|\xX(p)-\xX(\Psi^\pm_{t_2,t_1}(p))| &\leq |\xX(p)-\xX(p'')|\\
                                   &\leq |\xX(p)-\xX(p')|+|\xX(p')-\xX(p'')|\\
                                   &\leq 2 K^\prime (1+|\xX(p')|)^{-1}.
\end{align*}
As $K^\prime>0$ and $1+|\xX(p')|\geq 1$, \eqref{DecayEst} implies that
$$
(1+|\xX(p')|)^{-1}\leq (1+K^\prime)(1+|\xX(p)|)^{-1},
$$
and so
$$
|\xX(p)-\xX(\Psi^\pm_{t_2,t_1}(p))|\leq 2 K^\prime (1+K^\prime)(1+|\xX(p)|)^{-1}.
$$
Combining this with \eqref{SillyEst}
one obtains that for all $p \in \Gamma^\pm_{t_1}$,
$$
|\xX(p)-\xX(\Psi^\pm_{t_2,t_1}(p))|\leq \hat{K}(1+|\xX(p)|)^{-1}
$$
where $\hat{K}=2K^\prime (1+K^\prime)+\rho(1+2R)$.
By the same arguments, for all $p\in \Gamma^{+}_{-1}$,
$$
|\xX(p)-\xX(\Psi^{+,-}(p))|\leq \hat{K}(1+|\xX(p)|)^{-1}.
$$
Hence, it follows from Lemma \ref{StupidLem}, that
$$
|\xX(p)-\xX(\Psi(p))|\leq K(1+|\xX(p)|)^{-1}
$$
where $K=(1+\hat{K})^{2M+2}-1$.

To complete the proof set
$$
\psi(p)=\left\{\begin{array}{cc} \Pi^-(\Psi((\Pi^+)^{-1}(p))) & p \in \Omega_+ \\ p & p\in \partial \Omega_+. \end{array} \right. 
$$
We claim that $\psi$ is a homeomorphism. 
First, note that, by \cite[Item (3) of Proposition 4.4]{BernsteinWang2}, there is an $R>1$ and $\tilde{C}_1>1$ so that if $p\in \Gamma^\pm \backslash B_{R}$, then 
$$
\tilde{C}_1^{-1} |\xX(p)|^{2\mu} < \dist_{\Real^{n+1}}(p, \mathcal{C}(\Sigma)) < \tilde{C}_1 |\xX(p)|^{-1}
$$
where $\mu<-1$.  Hence, 
\begin{equation}\label{TwoSidedEst}
C^{-1} |\xX(p)|^{2\mu-1}<\dist_{\mathbb{S}^n}(\Pi^\pm (p), \mathcal{L}(\Sigma)) < C |\xX(p)|^{-2}
\end{equation}
where  $C\geq\tilde{C}_1$.
Hence, for $q\in \Omega^+$, with $\dist_{\mathbb{S}^n}(q, \mathcal{L}(\Sigma))$ sufficiently small, if we set $q'=(\Pi^+)^{-1}(q)\in \Gamma^+$, then
$$
|\xX(q')|\geq C^{\frac{1}{2\mu-1}} \dist_{\mathbb{S}^n}(q, \mathcal{L}(\Sigma))^{\frac{1}{2\mu-1}}.
$$
By \eqref{DistortionEst},
\begin{align*}
| |\xX(\Psi(q'))|-|\xX(q')|| &\leq 
|\xX(\Psi(q'))-\xX(q')|\\
 &\leq K C^{-\frac{1}{2\mu-1}} \dist_{\mathbb{S}^n}(q, \mathcal{L}(\Sigma))^{-\frac{1}{2\mu-1}}.
\end{align*}
Hence,  for   $\dist_{\mathbb{S}^n}(q, \mathcal{L}(\Sigma))$ sufficiently small,
$$
\dist_{\mathbb{S}^n}(q, \psi(q))\leq  4K C^{-\frac{1}{2\mu-1}} \dist_{\mathbb{S}^n}(q, \mathcal{L}(\Sigma))^{-\frac{1}{2\mu-1}} |\xX(q')|^{-1} .
$$
Using \eqref{TwoSidedEst}, again gives
$$
\dist_{\mathbb{S}^n}(q, \psi(q))\leq 4K C^{-\frac{2}{2\mu-1}} \dist_{\mathbb{S}^n}(q, \mathcal{L}(\Sigma))^{-\frac{2}{2\mu-1}}.
$$
As $\mu<-1$, for any $q_0\in \mathcal{L}(\Sigma)$, the right hand side goes to $0$ as $q\to q_0$. By the triangle inequality
$$
\dist_{\mathbb{S}^n}(q_0, \psi(q))\leq \dist_{\mathbb{S}^n}(q, \psi(q))+\dist_{\mathbb{S}^n}(q, q_0)
$$
and so the right hand side goes to $0$ as $q\to q_0$. Hence,  $\psi$ is continuous. Finally, as $\bar{\Omega}_+$ is compact and $\bar{\Omega}_-$ is Hausdorff, $\psi$ is a closed map and hence, as $\psi$ is a bijection, it is a homeomorphism.
\end{proof}

Theorem \ref{MainACSThm} is a standard topological consequence of Proposition \ref{MainACSProp}.  
\begin{proof} (of Theorem \ref{MainACSThm})

	  Observe that as $\mathcal{L}(\Sigma)$ is connected, by \cite[Theorem 0.1]{BernsteinWang2}, there are exactly two components of $\mathbb{S}^n\backslash \mathcal{L}(\Sigma)$, which we denote by $U^\pm$. Let $\phi:\mathbb{S}^n\to \mathbb{S}^n$ be the homeomorphism given by Proposition \ref{MainACSProp} so $\phi(U^-)=U^+$.  Pick a regular tubular neighborhood $T\subset \mathbb{S}^n$ of $\mathcal{L}(\Sigma)$.  We let $V^\pm =U^\pm \cup T$  and observe that $\bar{U}^\pm$, the closure of $U^\pm$, is a retract of $V^\pm$ and that $\mathcal{L}(\Sigma)$ is a retraction of $T=V^-\cap V^+$.  
	
	  As $\bar{U}^\pm$ is a retraction of $V^\pm$ and $\mathcal{L}(\Sigma)$ is a retraction of $T$, the natural inclusion maps induce isomorphisms between the reduced homology groups $\tilde{H}_k(\bar{U}^\pm)$ and $\tilde{H}_k(V^\pm)$ and between $\tilde{H}_k(\mathcal{L}(\Sigma))$ and $\tilde{H}_k(T)$. As such, there is a natural map $\Phi: \tilde{H}_k(V^-)\to \tilde{H}_k(V^+)$ defined by the following diagram,
	 	 \begin{equation*}
	 	 \begin{tikzcd}
	 	  {} & \tilde{H}_k(T) \arrow{r}{j^-_*} \arrow[pos=0.3]{dr}{j^+_*} & \tilde{H}_k(V^-) \arrow{d}{\Phi}\\
	 	 \tilde{H}_k(\mathcal{L}(\Sigma))  \arrow[leftrightarrow]{ur}{\simeq}\arrow{r}{i^-_*} \arrow{dr}{i^+_*} & \tilde{H}_k(\bar{U}^-)   \arrow[leftrightarrow, crossing over, pos=0.3]{ur}{\simeq} \arrow{d}{\phi_*} 	& \tilde{H}_k(V^+) \\
	 	 {} & \tilde{H}_k(\bar{U}^+)   \arrow[leftrightarrow]{ur}{\simeq}&
	 	 \end{tikzcd}
	 	 \end{equation*}
	  where  $i^\pm: \mathcal{L}(\Sigma)\to \bar{U}^\pm$ and $j^\pm:T\to V^\pm$ denote the natural inclusion maps and we used that $\phi\circ i^-=i^+$. As $\phi$ is a homeomorphism, both $\phi_*$ and $\Phi$ are  isomorphisms.  This implies that the map
  \begin{equation*}
   J=(j^-_*, -j^+_*):\tilde{H}_k(T) \to \tilde{H}_k(V^-)\oplus \tilde{H}_k(V^+)
  \end{equation*}
  is surjective if and only if $\tilde{H}_k(V^-)=\tilde{H}_k(V^+)=\set{0}$.  Indeed, if the map is surjective, then for any element $\alpha\in \tilde{H}_k(V^-)$ there is an element $\beta \in \tilde{H}_k(T)$ so that $J(\beta)=(\alpha,0)$.  That is,  $j_*^-(\beta)=\alpha$ and $j_*^+(\beta)=0$. Hence, $0=j_*^+(\beta)=\Phi(j_*^-(\beta))=\Phi(\alpha)$. In other words, as $\Phi$ is an isomorphism, $\alpha\in \ker(\Phi)=\set{0}$ and so $\tilde{H}_k(V^-)=\set{0}$. The proof that $\tilde{H}_k(V^+)=\set{0}$ is the same.  The converse is immediate.
  
  We next recall several standard facts about the reduced homology of manifolds and of manifolds with boundary.  First of all, as $\mathcal{L}(\Sigma)$ is a connected, oriented $(n-1)$-dimensional manifold, $\tilde{H}_k(\mathcal{L}(\Sigma))=\tilde{H}_k(T)=\set{0}$ for $k=0$ and $k\geq n$ and $\tilde{H}_{n-1}(\mathcal{L}(\Sigma))=\tilde{H}_{n-1}(T)=\mathbb{Z}$.  Likewise, as the $\bar{U}^\pm$ are connected, oriented $n$-manifolds with boundary, $\tilde{H}_k(\bar{U}^\pm)=\tilde{H}_k(V^\pm)=0$ for $k=0$ and $k\geq n$. 
  
  In order to compute the remaining reduced homology groups, we use the Mayer-Vietoris long exact sequence for the reduced homology of $(V^-,V^+,\mathbb{S}^n)$.  This gives the following exact sequences for $k\geq 0$
  \begin{equation}\label{ExactSeq}
  \begin{tikzcd}
  	\tilde{H}_{k+1}(\mathbb{S}^n)\arrow{r} & \tilde{H}_k(T) \arrow{r}{J} & \tilde{H}_k(V^-)\oplus \tilde{H}_k(V^+)
  \arrow{r}	& \tilde{H}_k(\mathbb{S}^n). 
  \end{tikzcd}
  \end{equation}
  As $\tilde{H}_k(\mathbb{S}^n)=\mathbb{Z}$ for $k=n$ and is otherwise $\set{0}$, \eqref{ExactSeq} implies  that $J$ is surjective for $0\leq k\leq n-1$.  Hence, for these $k$, $\tilde{H}_k(\bar{U}^\pm)=\tilde{H}_k(V^\pm)=\set{0}$ and so the $U^\pm$ are homology $n$-balls as claimed. As such, \eqref{ExactSeq} further implies that $\tilde{H}_k(\mathcal{L}(\Sigma))=\tilde{H}_k(T)=\set{0}$ for $0\leq k \leq n-2$ completing the verification that $\mathcal{L}(\Sigma)$ is a homology $(n-1)$-sphere. 
   
   To conclude the proof, it is enough, by the Hurewicz theorem, to show that $\pi_1(U^\pm)=\pi_1(\bar{U}^\pm)=\set{1}$.
   To that end first observe that the maps
   $F^\pm: \mathbb{S}^n\to \bar{U}^\pm$ defined by
   $$
   F^\pm(p)=\left\{ \begin{array}{ll} p & p\in \bar{U}^\pm \\ {\phi}(p) & p \in U^\mp \end{array} \right.
   $$
   are continuous. Now suppose $\gamma$ is a closed loop in $\bar{U}^\pm$.  As $\pi_1(\mathbb{S}^n)=\set{1}$, there is a homotopy $H: \mathbb{S}^1\times[0,1]\to \mathbb{S}^n$ taking $\gamma$ to a point.  Clearly, $F^\pm\circ H: \mathbb{S}^1\times[0,1]\to \bar{U}^\pm$ is also a homotopy taking $\gamma$ to a point. That is, $\pi_1(\bar{U}^\pm)=\set{1}$.
\end{proof}

\begin{proof}
(of Corollary \ref{ACSDim3Cor})

By Theorem \ref{MainACSProp}, $\mathcal{L}(\Sigma)$ is a homology $2$-sphere.  By the classification of surfaces this means that $\mathcal{L}(\Sigma)$ is diffeomorphic to $\mathbb{S}^2$ and so  Alexander's Theorem \cite{Alexander} implies that both components of $\mathbb{S}^3\backslash \mathcal{L}(\Sigma)$ are diffeomorphic to $\Real^3$, proving the claim.
\end{proof}

\section{Surgery Procedure} \label{SurgerySec}
We prove Theorem \ref{MainTopThm} using Corollary \ref{ACSDim3Cor} and Theorem \ref{CondSurgThm}. 
\begin{proof} (of Theorem \ref{MainTopThm})

We first observe that $\rstar{3,\lambda_2}$ holds by \cite[Theorem B]{MarquesNeves} and that $\rsstar{3,\lambda_2}$ holds by \cite[Corollary 1.2]{BernsteinWang2}.  If $\Sigma$ is (after a translation and dilation) a self-shrinker, then, by \cite[Theorem 0.7]{CIMW}, $\Sigma$ is diffeomorphic to $\mathbb{S}^3$, proving the theorem. Otherwise, flow $\Sigma$ for a small amount of time by the mean curvature flow (using short time existence of for smooth closed initial hypersurfaces) to obtain a hypersurface, $\Sigma^\prime$, diffeomorphic to $\Sigma$ and, by Huisken's monotonicity formula, with $\lambda[\Sigma']<\lambda[\Sigma]$.  On the one hand, if the level set flow of $\Sigma'$ is non-fattening, then we set $\Sigma_0=\Sigma'$. On the other hand, if the level set flow of $\Sigma'$ is fattening, then we can take $\Sigma_0$ to be a small normal graph over $\Sigma'$ so that $\lambda[\Sigma_0]<\lambda[\Sigma]$ and, because the non-fattening condition is generic, the level set flow of $\Sigma_0$ is non-fattening.

Hence, the hypotheses of Section \ref{SingularitySec} hold and we may apply Theorem \ref{CondSurgThm} unconditionally to obtain a family of hypersurfaces $\Sigma^1, \ldots, \Sigma^N$ in $\Real^4$. As $\Sigma^N$ is diffeomorphic to $\mathbb{S}^3$, if $N=1$, then there is nothing more to show and so we may assume that $N>1$.  We will now show that $\Sigma^{N-1}$ is diffeomorphic to $\Sigma^N$ and hence to $\mathbb{S}^3$. 

Let us denote by $V= \cup_{j=1}^{m(N-1)} V_j^{N-1}$ and by $\hat{\Sigma}^N=\Sigma^N\backslash V$ and let $U= \cup_{j=1}^{m(N-1)} U_j^{N-1}$ and $\hat{\Sigma}^{N-1}=\Sigma^{N-1}\backslash U$ so $\hat{\Phi}^{N-1}: \hat{\Sigma}^N \to \hat{\Sigma}^{N-1}$ is the orientation preserving diffeomorpism given by Theorem \ref{CondSurgThm}.  By Corollary \ref{ACSDim3Cor}, each component of $\bar{U}$ is diffeomorphic to a closed three-ball $\bar{B}^3$.  Hence, each component of $\partial \hat{\Sigma}^{N-1}$ and $\partial \hat{\Sigma}^{N}$ is diffeomorphic to $\mathbb{S}^2$.  That is,  for $1\leq j \leq m(N-1)$, $\partial V_j^{N-1}$ is diffeomorphic to $\mathbb{S}^2$ and so, as $\Sigma^N$ is diffeomorphic to the three-sphere,  Alexander's theorem \cite{Alexander} implies that each $\bar{V}_j^{N-1}$ is diffeomorphic to the closed three-ball.  Hence, there are orientation preserving diffeomorphisms $\Psi_j^{N-1}: \bar{V}_j^{N-1}\to \bar{U}_j^{N-1}$. 

Denote by $\hat{\phi}^{N-1}_j: \partial V_j^{N-1} \to \partial U_j^{N-1}$ the diffeomorphism given by restricting $\hat{\Phi}^{N-1}$ and, likewise, let $\psi^{N-1}_j: \partial V_j^{N-1}\to \partial U_j^{N-1}$ denote the diffeomorphisms given by restricting $\Psi^{N-1}_j$.  Observe, that the orientation of $\hat{\Sigma}^{N}$ and the orientation on $\bar{V}$ induce opposite orientations on $\partial \bar{V}$.  Likewise, the orientation of $\hat{\Sigma}^{N-1}$ and that of $\bar{U}$ induce opposite orientations on $\partial \bar{U}$.   By construction, the $\hat{\phi}^{N-1}_j$ preserve the orientations induced from $\hat{\Sigma}^N$ and $\hat{\Sigma}^{N-1}$. Hence, as the orientations induced by $\bar{V}_j^{N-1}$ and $\bar{U}_j^{N-1}$ are opposite to those induced by $\hat{\Sigma}^N$ and $\hat{\Sigma}^{N-1}$, the $\hat{\phi}^{N-1}_j$ also preserve these orientations.  The same is true of the $\psi^{N-1}_j$.  As such, $\xi_j^{N-1}=(\psi_{j}^{N-1})^{-1}\circ \hat{\phi}_j^{N-1}\in \mathrm{Diff}_+(\partial V_j^{N-1})$, where $\mathrm{Diff}_+(M)$ is the space of orientation preserving self-diffeomorphisms of an oriented manifold $M$ (here we may use the orientation on $\partial V_j^{N-1}$ induced by either $\bar{V}$ or $\hat{\Sigma}^N$).  By \cite{Munkres} -- see also \cite{Smale} and \cite{Cerf} -- the space $\mathrm{Diff}_+(\mathbb{S}^2)$ is path-connected and so any element of $\mathrm{Diff}_+(\mathbb{S}^2)$ extends to an element of $\mathrm{Diff}_+(\bar{B}^3)$.  That is, there are diffeomorphism $\Xi_j^{N-1}\in \mathrm{Diff}_+( \bar{V}_j^{N-1})$ that restrict to $\xi_{j}^{N-1}$ on $\partial  V_j^{N-1}$.  Thus, the maps $\hat{\Psi}_j^{N-1}=\Psi_j^{N-1}\circ \Xi_{j}^{N-1}: \bar{V}_j^{N-1} \to \bar{U}_j^{N-1}$ are diffeomorphisms that agree with $\hat{\Phi}^{N-1}$ on the common boundary.

Define $\Phi^{N-1}:\Sigma^N \to \Sigma^{N-1}$ by
$$
\Phi^{N-1}(p)=\left\{ \begin{array}{cc} \hat{\Phi}^{N-1}(p) & p \in \hat{\Sigma}^N \\ \hat{\Psi}_j^{N-1}(p) & p\in V_j^{N-1}. \end{array} \right.
$$
By construction, this map is a homeomorphism.  However, it is a standard procedure to construct a diffeomorphism between $\Sigma^N$ and $\Sigma^{N-1}$ by smoothing this map out (see for instance \cite[Theorem 8.1.9]{Hirsch}).  Hence, $\Sigma^{N-1}$ is diffeomorphic to $\mathbb{S}^3$ and iterating this argument shows that $\Sigma=\Sigma^1$ is diffeomorphic to $\mathbb{S}^3$ as claimed.
\end{proof}

Theorem \ref{MainCondThm} follows from Theorem \ref{MainACSThm}, Theorem \ref{CondSurgThm} and the Mayer-Vietoris long exact sequence for reduced homology.  For completeness, we include a proof of the following standard topological fact which we will need to use.
\begin{lem} \label{HomBallLem}
Let $M$ be a closed $n$-dimensional manifold and $\Sigma \subset M$ a closed hypersurface.  If $M$ is a homology $n$-sphere and $\Sigma$ is a homology $(n-1)$-sphere, then each component of $M\backslash \Sigma$ is a homology $n$-ball.
\end{lem}
\begin{proof}
Our hypotheses ensure that both $M$ and $\Sigma$ are connected and oriented.  Hence, $\Sigma$ is two-sided and there is an open $U^+\subset M$ so that $\Sigma =\partial U^+$.  Let $U^-=M\backslash \bar{U}^+$. To prove the lemma we will need to compute the Mayer-Vietoris long exact sequence for $(\bar{U}^-, \bar{U}^+, M)$.   Strictly speaking, we should ``thicken" $\bar{U}^+$ and $\bar{U}^-$ up with a regular tubular neighborhood of $\Sigma=\partial \bar{U}^\pm$ as in the proof of Theorem \ref{MainACSThm}, but we leave the details of this to the reader. 

 The Mayer-Vietoris long exact sequence and the fact that $M$ is a homology $n$-sphere and $\Sigma$ is a homology $(n-1)$-sphere gives the sequences
\begin{equation*}
   \begin{tikzcd}
   	\tilde{H}_{k+1}(M)\arrow{r}{\partial} \arrow[leftrightarrow]{d}{=} & \tilde{H}_k(\Sigma) \arrow{r} \arrow[leftrightarrow]{d}{=}& \tilde{H}_k(\bar{U}^-)\oplus \tilde{H}_k(\bar{U}^+)
   \arrow{r} \arrow[leftrightarrow]{d}{=}	& \tilde{H}_k(M) \arrow[leftrightarrow]{d}{=}\\
   \tilde{H}_{k+1}(\mathbb{S}^n) \arrow{r}{\partial} & \tilde{H}_k(\mathbb{S}^{n-1}) \arrow{r} & \tilde{H}_k(\bar{U}^-)\oplus \tilde{H}_k(\bar{U}^+)
      \arrow{r}	& \tilde{H}_k(\mathbb{S}^{n}).   
   \end{tikzcd}
 \end{equation*}
 For $0\leq k\leq n-2$ and $k\geq n+1$ this immediately gives that $\tilde{H}_k(\bar{U}_\pm)=\set{0}$.  When $k={n-1}$, the map $\partial$ is necessarily generated by $[M]\mapsto [\Sigma]$ where $[M] $ is the fundamental class of $M$ and $[\Sigma]$ is the fundamental class of $\Sigma$. In particular, this map is an isomorphism and so we conclude that $\tilde{H}_{n-1}(\bar{U}^\pm)=\set{0}$.    For the same reason, $\tilde{H}_n(\bar{U}^\pm)=\set{0}$, which verifies the claim.  
\end{proof}
 
\begin{proof}(of Theorem \ref{MainCondThm})

Arguing as in the first paragraph of the proof of Theorem \ref{MainTopThm}, we obtain $\Sigma^1,\ldots, \Sigma^N$ the hypersurfaces given by Theorem \ref{CondSurgThm}. As $\Sigma^N$ is diffeomorphic to $\mathbb{S}^n$, it is a homology $n$-sphere. In particular, if $N=1$, then there is nothing further to show.  As such, we may assume that $N>1$. 

 Let us show that $\Sigma^{N-1}$ is a homology $n$-sphere.  First, set $V= \cup_{j=1}^{m(N-1)} V_j^{N-1}$ and $\hat{\Sigma}^N=\Sigma^N\backslash V$ and let $U= \cup_{j=1}^{m(N-1)} U_j^{N-1}$ and $\hat{\Sigma}^{N-1}=\Sigma^{N-1}\backslash U$.  Next observe that, as $\partial U_j^{N-1}=\mathcal{L}(\Gamma_j^{N-1})$ for some $\Gamma_j^{N-1}\in \mathcal{ACS}_n^*(\Lambda)$, Theorem  \ref{MainACSThm} implies that each component of $\partial \hat{\Sigma}^{N-1}$ is a homology $(n-1)$-sphere.  Hence, as $\partial U=\partial \hat{\Sigma}^{N-1}$ is diffeomorphic to $\partial \hat{\Sigma}^N=\partial V$, we see that each component of $\partial V=\partial \hat{\Sigma}^N$ is a homology $(n-1)$-sphere and so Lemma \ref{HomBallLem} implies that each component of $\bar{V}$ is a homology $n$-ball.
 
We may now use the Mayer-Vietoris long exact sequence to compute that $\tilde{H}_k(\hat{\Sigma}^N)=\set{0}$ for $k\neq n-1$ and $\tilde{H}_{n-1}(\hat{\Sigma}^N)=\mathbb{Z}^{m(N-1)-1}$. To see this, consider the Mayer-Vietoris long exact sequence of $(\bar{V}, \hat{\Sigma}^N, \Sigma^N)$.  This long exact sequence and the fact that $\bar{V}$ is the union of homology $n$-balls gives, for $k>0$, the exact sequences
 \begin{equation*}
  \begin{tikzcd}
  	\tilde{H}_{k+1}(\Sigma^N)\arrow{r}{\partial} \arrow[leftrightarrow]{d}{=} & \tilde{H}_k(\partial V) \arrow{r} \arrow[leftrightarrow]{d}{=}& \tilde{H}_k(\bar{V})\oplus \tilde{H}_k(\hat{\Sigma}^N)
  \arrow{r}\arrow[leftrightarrow]{d}{=}	& \tilde{H}_k(\Sigma^N)\arrow[leftrightarrow]{d}{=}\\
  \tilde{H}_{k+1}(\mathbb{S}^n)\arrow{r}{\partial}  & \bigoplus\limits_{j=1}^{m(N-1)}\tilde{H}_k(\mathbb{S}^{n-1}) \arrow{r} & \tilde{H}_k(\hat{\Sigma}^N)
    \arrow{r}	& \tilde{H}_k(\mathbb{S}^n).
  \end{tikzcd}
   \end{equation*}
Hence, for $1\leq k\leq n-2$ and $k\geq n+1$, $\tilde{H}_k(\hat{\Sigma}^N)=\set{0}$.  When $k=n-1$, the map $\partial$ is generated by $ [\Sigma^N]\mapsto([\partial V_{1}^{N-1}],  \ldots, [\partial V_{m(N-1)}^{N-1}])$ where $[\Sigma^N]$ is the fundamental class of $\Sigma^N$ and $[\partial V_{j}^{N-1}]$ is the fundamental class of $\partial V_{j}^{N-1}$.  It follows that $ \tilde{H}_{n-1}(\hat{\Sigma}^N)=\mathbb{Z}^{m(N-1)-1}$ and, as this map is injective, that $\tilde{H}_{n}(\hat{\Sigma}^N)=\set{0}$.
 Finally,  as $\hat{\Sigma}^N$ is connected, $\tilde{H}_0(\hat{\Sigma}^N)=\set{0}$, which completes the computation.

By Theorem \ref{CondSurgThm}, $\hat{\Sigma}^N$ is diffeomorphic to $\hat{\Sigma}^{N-1}$ and so $\tilde{H}_k(\hat{\Sigma}^{N-1})=0$ for $k\neq n-1$ and $\tilde{H}_{n-1}(\hat{\Sigma}^{N-1})=\mathbb{Z}^{m(N-1)-1}$.  Furthermore, Theorem \ref{MainACSThm} implies that each component of $\bar{U}$ is contractible. Hence, applying the Mayer-Vietoris long exact sequence to $(\hat{\Sigma}^{N-1}, \bar{U}, \Sigma^{N-1})$ gives, for $k>0$,
  \begin{equation*}
   \begin{tikzcd}[column sep=12pt]
   \tilde{H}_k(\partial \bar{U}) \arrow{r} \arrow[leftrightarrow]{d}{=} & \tilde{H}_k(\bar{U})\oplus \tilde{H}_k(\hat{\Sigma}^{N-1})
   \arrow{r}	\arrow[leftrightarrow]{d}{=}& \tilde{H}_k(\Sigma^{N-1})  \arrow{r}\arrow[leftrightarrow]{d}{=} &\tilde{H}_{k-1}(\partial \bar{U})\arrow[leftrightarrow]{d}{=}\\
   \bigoplus\limits_{j=1}^{m(N-1)}\tilde{H}_k(\mathbb{S}^{n-1}) \arrow{r} & \tilde{H}_k(\hat{\Sigma}^{N-1}) \arrow{r} & \tilde{H}_k(\Sigma^{N-1})\arrow{r} & \bigoplus\limits_{j=1}^{m(N-1)} \tilde{H}_{k-1} (\mathbb{S}^{n-1}).
   \end{tikzcd}
  \end{equation*} 
  In particular, for $1\leq k\leq n-2$ and $k\geq n+1$, we obtain that $\tilde{H}_k(\Sigma^{N-1})=\set{0}$.  The Mayer-Vietoris long exact sequence further gives the exact sequences
  \begin{equation*}
    \begin{tikzcd}
     \tilde{H}_{n-1}(\partial \bar{U}) \arrow{r}{\delta}  \arrow[leftrightarrow]{d}{=}& \tilde{H}_{n-1}(\bar{U} )\oplus \tilde{H}_{n-1}(\hat{\Sigma}^{N-1})   \arrow{r} \arrow[leftrightarrow]{d}{=} & \tilde{H}_{n-1}(\Sigma^{N-1})  \arrow{r} \arrow[leftrightarrow]{d}{=} &\tilde{H}_{n-2}(\partial \bar{U}) \arrow[leftrightarrow]{d}{=}\\ 
    \mathbb{Z}^{m(N-1)} \arrow{r}{\delta} & \mathbb{Z}^{m(N-1)-1}
    \arrow{r}	& \tilde{H}_{n-1}(\Sigma^{N-1})  \arrow{r} &\set{0}. 
    \end{tikzcd}
   \end{equation*} 
   Here  $\delta$ is given by $(l_1, \ldots, l_{m(N-1)})\mapsto (l_1-l_{m(N-1)}, \ldots, l_{m(N-1)-1}-l_{m(N-1)})$. As $\delta$ is surjective, it follows that $\tilde{H}_{n-1}(\Sigma^{N-1})=\set{0}$.
 Finally, as $\Sigma^{N-1}$ is an oriented, connected $n$-dimensional manifold $\tilde{H}_n(\Sigma^{N-1})=\mathbb{Z}$ and $\tilde{H}_0(\Sigma^{N-1})=\set{0}$.   Hence, $\Sigma^{N-1}$ is a homology $n$-sphere.

As our argument only used that ${\Sigma}^N$ was a homology $n$-sphere, we may repeat it to see that each of the ${\Sigma}^i$ is a homology $n$-sphere and so conclude that $\Sigma$ is one as well.
\end{proof}

\appendix
\section{}
Fix an open subset $U\subset \Real^{n+1}$.  A \emph{hypersurface in $U$}, $\Sigma$, is a proper, codimension-one submanifold of $U$. A \emph{smooth mean curvature flow in $U$}, $S$, is a collection of hypersurfaces in $U$, $\set{\Sigma_t}_{t\in I}$, $I$ an interval, so that:
\begin{enumerate}
\item For all $t_0\in I$ and $p_0\in \Sigma_{t_0}$ there is a $r_0=r_0(p_0,t_0)$ and an interval $I_0=I_0(p_0,t_0)$ with $(p_0,t_0)\in B_{r_0}^{n+1}(p_0)\times I_0\subset U\times I$;
\item There is a smooth map $F: B_{1}^n\times I_0\to \Real^{n+1}$
so that $F_t(p)=F(p,t): B_1^n\to \Real^{n+1}$ is a parameterization of $B_{r_0}^{n+1}(p_0)\cap \Sigma_t$; and
\item $\left(\frac{\partial}{\partial t} F (p,t)\right)^\perp= \mathbf{H}_{\Sigma_t}(F(p,t)).$
\end{enumerate}
It is convenient to consider the \emph{space-time track} of $S$ (also denoted by $S$):
\begin{equation}
{S}=\set{(\xX(p),t)\in\mathbb{R}^{n+1}\times \mathbb{R}: p\in\Sigma_t}\subset U\times I.
\end{equation}
This is a smooth submanifold of space-time and is transverse to each constant time hyperplane $\Real^{n+1}\times \set{t_0}$.  Along the space-time track ${S}$,  let $\frac{d}{dt}$ be the smooth vector field
\begin{equation}
\left.\frac{d}{dt}\right|_{(p,t)}=\frac{\partial}{\partial t}+\mathbf{H}_{\Sigma_t}(p).
\end{equation}
It is not hard to see that this vector field is tangent to ${S}$ and the position vector satisfies
\begin{equation} \label{MCFEqn}
\frac{d}{dt} \xX(p,t)=\mathbf{H}_{\Sigma_t}(p).
\end{equation}

It is a standard fact that if each $\Sigma_t$ in $S$ is closed, i.e. is compact and without boundary, then there is a smooth map
$$
F:M\times I \to \Real^{n+1}
$$
so that each $F_t=F(\cdot,t): M\to \Real^{n+1}$ is a parameterization of $\Sigma_t$ a closed $n$-dimensional manifold $M$. As a consequence, each $\Sigma_t$ is diffeomorphic to $M$.

We will need the following generalization of this last fact to manifolds with boundary. 
\begin{prop} \label{BoundaryDiffProp}
    Fix $R\in (0,\infty]$ and let $\set{\bar{B}_{2r_1}(\xX_1), \ldots, \bar{B}_{2r_m}(\xX_m)}$ be a collection of pairwise disjoint balls in $B_{R}\subset \Real^{n+1}$ and let $U=B_{2R}\backslash \bigcup_{i=1}^m \bar{B}_{r_i}(\xX_i)$. 
	If $\set{\Sigma_{t}}_{t\in (-\tau,\tau)}$ is a smooth mean curvature flow in $U$ with the property that
	\begin{enumerate}
	\item Each $\hat{\Sigma}_t=\Sigma_t\cap \left(\bar{B}_R\backslash 
	\bigcup_{i=1}^m B_{2 r_i}(\xX_i)\right)$ is compact,
	\item For each $1\leq i \leq m$, $\partial B_{2 r_i} (\xX_i)$ intersects $\Sigma_t$ transversally and non-trivially for all $t\in (-\tau,\tau)$,
	\item If $R<\infty$, then $\partial B_R$ intersects $\Sigma_t$ transversally and non-trivially for all $t\in (-\tau,\tau)$,
	\end{enumerate}
	then, for any $t_1,t_2\in (-\tau,\tau)$, $\hat{\Sigma}_{t_1}$ and $\hat{\Sigma}_{t_2}$ are diffeomorphic as compact manifolds with boundary.
\end{prop}
\begin{proof}
    For simplicity, we consider only $R=\infty$, $m=1$, $\xX_1=\OO$ and $r_1=\frac{1}{2}$.  It is straightforward to extend the argument to the general case.
	Let $S$ be the space-time track of the flow, so $S$ is a smooth hypersurface in $(\Real^{n+1}\backslash \bar{B}_{1/2})\times(-\tau,\tau)$.   As each $\Sigma_t$ intersects $\partial B_1$ transversally, it is clear that $S$ meets $\partial B_1 \times (-\tau,\tau)$ transversally.  In particular, $\tilde{S}=S\backslash \left({B}_1 \times (-\tau,\tau)\right)$ is a smooth hypersurface with boundary.  Let $\tilde{B}=\partial \tilde{S}=\set{(p,t):p\in \partial B_1\cap \Sigma_t, t\in (-\tau,\tau)}$.   
	
	Without loss of generality we may assume that the given $t_1,t_2$ satisfy $t_1<t_2$. Let $\hat{S}= \tilde{S}\cap \left(\Real^{n+1}\times [t_1,t_2]\right)$ and $\hat{B}=\tilde{B}\cap \left(\Real^{n+1}\times [t_1,t_2]\right)$.  Observe that $\hat{S}$ is a compact manifold with corners and $\hat{B}$ is one of its boundary strata.  The other two boundary strata are $\hat{\Sigma}_{t_1}\times \set{t_1}$ and $\hat{\Sigma}_{t_2}\times \set{t_2}$.
	
	 As $\partial B_1$ meets each $\Sigma_t$ transversally and $\hat{B}$ is compact, there is an $\epsilon>0$ so that, for $(p,t)\in \hat{B}$, $|\xX^\top(p,t)|\geq 2\epsilon$, where $\xX^\top$ is the tangential component of the position vector.  By continuity there is a $\frac{1}{2}>\delta>0$ so that, for any $t\in [t_1,t_2]$ and $p\in\left(\bar{B}_{1+\delta} \backslash B_{1-\delta}\right)\cap \Sigma_t$,  $|\xX^\top(p,t)|\geq \epsilon$. 
  Now let $\eta \in C^\infty_0(\Real^{n+1})$  be a smooth function with $0\leq \eta\leq 1$,  $\eta=1$ on $\partial B_1$ and $\spt(\eta)\subset\bar{B}_{1+\delta} \backslash B_{1-\delta}$.
  For $(p,t)\in \hat{S}$ consider the vector 
   $$
   \mathbf{V}(p,t)= - \eta(\xX(p,t)) \frac{ (\xX(p,t)\cdot \mathbf{H}_{\Sigma_t}(p))}{|\xX^\top(p,t)|^2} \xX^\top(p,t) 
   $$
and observe this gives a smooth vector field on $S$ that restricts to a smooth compactly supported vector field on each $\Sigma_t$.
Let $\mathbf{W}=\frac{d}{dt}+\mathbf{V}$ which is a smooth vector field on ${S}$.

We claim that $\mathbf{W}$ is tangent to $\hat{B}$ and transverse to $\hat{\Sigma}_{t_1}\times\set{t_1}\cup \hat{\Sigma}_{t_2}\times \set{t_2}$. As $\mathbf{V}$ is tangent to $\Sigma_t\times\set{t}$, the transversality of $\mathbf{W}$ follows from the transversality of $\frac{d}{dt}$. This transversality follows immediately from the definition of $\frac{d}{dt}$. To see the tangency note that, by construction, $\hat{B}=\set{(p,t)\in \hat{S}: |\xX(p,t)|^2=1}$.  For $(p,t)\in \hat{B}$, one computes
\begin{align*}
\mathbf{W}\cdot |\xX(p,t)|^2& = 2 \xX(p,t) \cdot \nabla_\mathbf{W} \xX(p,t)\\
&=2\xX(p,t) \cdot \mathbf{H}_{\Sigma_t}(p)-2\eta(\xX(p,t)) \frac{ (\xX(p,t)\cdot \mathbf{H}_{\Sigma_t}(p))}{|\xX^\top(p,t)|^2} \xX(p,t)\cdot \xX^\top(p,t) \\
&=0
\end{align*}
where the last equality used that $(p,t)\in \hat{B}$ so $\eta(\xX(p,t))=1$.   This verifies the claim.

To conclude the proof observe that, as $\hat{S}$ is compact and $\mathbf{W}$ is tangent to $\hat{B}$ and transverse to $\hat{\Sigma}_{t_1}\times\set{t_1}\cup \hat{\Sigma}_{t_2}\times\set{t_2}$, standard ODE theory gives that for any $P_0=(p_0,t_0)\in \hat{S}$  the initial value problem 
$$
\left\{ \begin{array}{c} \dot{\gamma}(s)=\mathbf{W}(\gamma(s)) \\ \gamma_{P_0}(0)=P_0 \end{array}
\right.
$$
has a unique smooth solution $\gamma_{P_0}:[t_1-t_0,t_2-t_0]\to \hat{S}$ which depends smoothly on $P_0$.  These solutions satisfy $t(\gamma_{P_0}(s))=s+t_0$ and so there is a diffeomorphism $\phi:\Sigma_{t_1}\to \Sigma_{t_2}$ given by $(\phi(p),t_2)=\gamma_{(p,t_1)}(t_2-t_1)$.
\end{proof}

\end{document}